\documentclass[final]{jcmlatex}

\def\JCMvol{xx}
\def\JCMno{x}
\def\JCMyear{2026}
\def\JCMreceived{}
\def\JCMrevised{}
\def\JCMaccepted{}
\usepackage{amsmath}
\usepackage{amsfonts}
\usepackage{inconsolata}
\usepackage[T1]{fontenc}
\usepackage{mathrsfs}
\usepackage{xcolor}
\usepackage{graphicx}
\usepackage{hyperref}
\usepackage{algorithmicx,algorithm}
\usepackage[noend]{algpseudocode}
\usepackage[caption = false]{subfig}
\usepackage{booktabs}
\usepackage{enumitem}
\usepackage{bm}
\usepackage{upgreek}
\usepackage{blkarray}
\usepackage{nicematrix}
\usepackage{tikz} 
\usetikzlibrary{decorations.pathreplacing, calligraphy}
\usepackage{tabularx}
\usepackage{booktabs}
\usepackage{longtable}

\AtBeginDocument{%
  \paperwidth=\dimexpr
    1in + \oddsidemargin
    + \textwidth
    + 1in + \oddsidemargin
  \relax
  \paperheight=\dimexpr
    1in + \topmargin
    + \headheight + \headsep
    + \textheight
    + 1in + \topmargin
  \relax
  \usepackage[pass]{geometry}\relax
}

\newcommand{\qed}{\hfill\quad\rule{0.5em}{0.5em}}

\begin{document}

\markboth{WENHAO LI, ZONGYUAN HAN, SHENGXIN ZHU}{On Symmetric Lanczos Quadrature for Stochastic Trace Estimation}

\title{ON SYMMETRIC LANCZOS QUADRATURE FOR STOCHASTIC TRACE ESTIMATION}


%    author one information

\author{Wenhao Li
\affiliation{Department of Mathematics, Hong Kong Baptist University, Hong Kong, 999077, China \\
Email: liwenhao@bnbu.edu.cn}
\and
Shengxin Zhu
\affiliation{Beijing Normal University, Zhuhai 519087, China \\
Email: Shengxin.Zhu@bnu.edu.cn}}
          
\maketitle

%    Abstract is required.
\begin{abstract}
    A common approach to approximating quadratic forms of matrix functions is to use a quadrature rule derived from the Lanczos process, known as a Lanczos quadrature. Although symmetric quadrature rules are computationally favorable, it has remained unclear whether a symmetric Lanczos quadrature is practically feasible. In this work, we resolve this ambiguity by establishing necessary and sufficient conditions for the existence of symmetric Lanczos quadratures. We show that the sufficient condition can be met for a class of Jordan-Wielandt matrices by carefully constructing initial vectors with specific distributions for the Lanczos algorithm. Applying such a symmetric Lanczos quadrature to compute the Estrada index of bipartite or directed graphs ensures that the resulting stochastic trace estimators are unbiased. Furthermore, we observe that the variance of the quadratic form estimator based on the symmetric Lanczos quadrature is lower than that of the standard estimator.
    
\end{abstract}



\section{Introduction}
\label{sec:background}

Given a smooth function $f$ defined on $[a, b]$ and a non-negative weight function $\omega$, the Golub-Welsch algorithm \cite{GW69} produces a Gaussian quadrature rule,
$$\int_{a}^{b} f(t) \omega(t) dt \approx \sum_{k=1}^m \tau_k f(\theta_k), $$
where quadrature nodes $\{ \theta_k\}_{k=1}^m$ are the eigenvalues of the tridiagonal \textit{Jacobi matrix} produced in the Lanczos process \cite{L50} and the quadrature weights are the squares of the first element of each normalized eigenvector of the Jacobi matrix. Such a quadrature rule based on the Lanczos method plays a central role in estimating quadratic forms involving matrix functions  
\begin{equation}
      \label{eq:GQ}
    Q(\textbf{u},f,\textbf{A}) = \textbf{u}^T f(\textbf{A}) \textbf{u} \approx \Arrowvert \textbf{u}\Arrowvert_2^2 \sum_{k=1}^m \tau_k f(\theta_k),
\end{equation}
where $\textbf{A}$ is a symmetric matrix and where $\textbf{u}$ is a vector, see \cite{GM94,GS94} or the following sections for more details. This method was previously used in the analysis of iterative methods \cite{BG99,CGR99,F92,M97} \cite[p.195]{M99} and the finite element method \cite{R92}. Bai, Golub and Fahey applied this method within the realm of quantum chromodynamics (QCD) \cite{BFG96}. Besides, it can be used to measure the centrality of complex networks. For example, subgraph centrality can be quantified as $\textbf{e}_i^T e^{\gamma A} \textbf{e}_i$ \cite{ER05}, whereas resolvent-based subgraph centrality is associated with $\textbf{e}_i^T(\textbf{I} - \gamma \textbf{A})^{-1} \textbf{e}_i$ \cite{EH10}, where $\textbf{e}_i$ is the $i^{th}$ column of the identity matrix. Rapid calculation of centrality metrics is crucial for real-world applications, such as optimizing urban public transportation networks \cite{BS21,BS22}. In genomics, statisticians often require approximations of the distribution of quadratic forms \eqref{eq:GQ} with normally distributed vectors \cite{CL19,BQF,LBPNR18}. In Gaussian process regression, hyperparameter optimization is achieved through the maximization of a likelihood function, which requires the estimation of quadratic forms \eqref{eq:GQ} \cite{DE17,SA22,zhu18}. 
The Lanczos quadrature method offers effective computational means for accurately computing these quantities. Furthermore, this method can be integrated into Krylov spectral methods to solve time-dependent partial differential equations \cite{L05,L07}. 

Recent studies \cite{CK21,UCS17} have proposed frameworks that integrate the Lanczos method, quadrature rules, and Monte Carlo techniques to estimate the trace of matrix functions \cite{AT11,H90}, leveraging a series of quadratic forms with randomly generated vectors. Both follow the stochastic Lanczos quadrature method. One gives error analysis of symmetric Lanczos quadrature on approximating the Riemann-Stieltjes integral, while the other focuses on the error of asymmetric Lanczos quadrature. 

Let us first define the symmetric Lanczos quadrature.
\begin{definition}
    \label{def:symlq}
    Let $\textbf{A}\in \mathbb{R}^{n \times n}$ be a symmetric matrix, and let $\textbf{u}\in\mathbb{R}^n$ be an initial vector for the $m$-step Lanczos iteration, where $m \le n$. Then the corresponding $m$-node Lanczos quadrature is \textbf{symmetric} if for all $j = 1,\ldots,m$,
    \begin{itemize}
        \item quadrature nodes $\{\theta_k\}_{k=1}^j$ are symmetric about $0$, and
        \item symmetric nodes (about $0$) possess identical quadrature weights, i.e., $\tau_{p}=\tau_q$ if $\theta_p = -\theta_q$, $p,q\le j$.
    \end{itemize}
\end{definition}
An $m$-node Lanczos quadrature is \textbf{asymmetric} if at least one property in Definition \ref{def:symlq} does not hold. Ubaru, Chen, and Saad presented the upper bound of the quadrature error when the Gauss quadrature rule is symmetric (\cite[Section 4.1]{UCS17}, \cite{R69}). Few years later, Cortinovis and Kressner remarked that the integral of odd-degree Chebyshev polynomials associated with the measure (in Lanczos iterations) is actually not zero by symmetry, thus proposing a different upper bound \cite[Section 3]{CK21}. Comparative analysis of the two results reveals that a symmetric Lanczos quadrature is more computationally favorable. Similarly, \cite{MS14} showed that the symmetry of weight functions, such as Gegenbauer weights or Hermite weights, can be utilized to speed up quadrature computations. Then an intriguing inquiry arises regarding the circumstances under which the Lanczos algorithm would yield symmetric quadrature nodes and weights in exact arithmetic for better computational performance. Furthermore, we wonder how to ensure symmetric Lanczos quadrature in practice, e.g., in stochastic trace estimation.

To the best of our knowledge, over the past few years, researchers have developed and analyzed the Lanczos-based algorithms in spectrum approximation \cite{CTU21,CTU24}, matrix function estimation \cite{JL15,CGMM22}, and trace estimation \cite{CH23,E23,M21,P22}. There are also studies investigating spatial distribution \cite{B13,C11,CE12,M15}, convergence characteristics \cite{W05,ELM21,TM12}, and stabilization techniques \cite{MS06,P71,P80} for the Ritz values (eigenvalues of the tridiagonal Jacobi matrix obtained within the Lanczos method or Arnoldi process). However, none of these methods well studies the symmetry of Lanczos quadrature, or even its applications in stochastic trace estimation. To fill the theoretical gap, this work derives a necessary (Theorem \ref{thm:nec}) and sufficient (Theorem \ref{Thm:main}) condition for $m$-node symmetric Lanczos quadrature rules. Specifically, for Jordan-Wielandt matrices with $\textbf{B} \in \mathbb{R}^{n_1 \times n_2}$
\begin{equation}
    \label{Eq:A}
    \textbf{A} = \begin{bmatrix}
         & \textbf{B} \\
        \textbf{B}^H & 
    \end{bmatrix} \in \mathbb{R}^{(n_1+n_2) \times (n_1+n_2)},
\end{equation}
that have symmetric spectra around $0$ \cite{B96,S03}, we show one can construct a certain type of initial vectors to guarantee a symmetric Lanczos quadrature without information of the matrix rank (Theorem \ref{prop:2}). In applications that compute the Estrada index, we propose a modified unbiased trace estimator for matrix exponentials. Numerical simulations demonstrate the higher convergence based on symmetric quadrature rules.

The paper is organized as follows.  We start with some preliminary results related to the Lanczos quadrature method in Section \ref{sec:basics} to make this work self-contained. In Section \ref{sec:EA}, we propose sufficient and necessary conditions for the Lanczos quadrature with symmetric quadrature nodes and weights. Furthermore, for approximating the Estrada index of directed graphs and bipartite graphs in complex networks, we propose unbiased trace estimators that utilize the symmetry of the Lanczos quadrature. Numerical experiments are given in Section \ref{sec:experiments} to illustrate our theoretical results on the basis of synthetic matrices and real applications, and conclusions are presented in Section \ref{sec:cr}. Table \ref{tab:notation} summarizes all the symbols used in this paper.

\section{Lanczos quadrature method}
\label{sec:basics}
For a symmetric\footnote{The Lanczos process does not require positive definiteness, but in some applications such a property is needed.} matrix $\textbf{A}$, its quadratic form $Q(\textbf{u},f,\textbf{A})$ with vector $\textbf{u}$ and matrix function $f$ can be estimated using the Lanczos quadrature method \cite{GM94,GS94}. Let $\textbf{A} = \textbf{Q}\mathbf{\Lambda} \textbf{Q}^T$ be the eigen-decomposition with $\mathbf{\Lambda} = \mathrm{diag}(\lambda_1, \lambda_2, \cdots, \lambda_n)$ and $\lambda_{\min} = \lambda_1 \le \lambda_2 \le \ldots \le \lambda_n = \lambda_{\max}$ without loss of generality. Then, $f(\mathbf{\Lambda})$ is a diagonal matrix with entries $\{f(\lambda_j)\}_{j=1}^n$ and $f(\textbf{A}) = \textbf{Q}f(\mathbf{\Lambda})\textbf{Q}^T$. Let $\textbf{v}$ be the normalized vector $\textbf{v} = \textbf{u}/\Arrowvert\textbf{u}\Arrowvert_2$ and $\boldsymbol{\mu} = \textbf{Q}^T \textbf{v}$, then $Q(\textbf{u},f,\textbf{A})$ is represented as
\begin{equation}
    \label{eq:qfmu}
    Q(\textbf{u},f,\textbf{A}) = \textbf{u}^Tf(\textbf{A})\textbf{u} = \Arrowvert\textbf{u}\Arrowvert_2^2 \textbf{v}^T \textbf{Q} f(\mathbf{\Lambda})\textbf{Q}^T\textbf{v} = \Arrowvert\textbf{u}\Arrowvert_2^2 \boldsymbol{\mu}^Tf(\mathbf{\Lambda})\boldsymbol{\mu}. 
\end{equation}
The last quadratic form in \eqref{eq:qfmu} can equivalently be considered as the sum $\sum_{j=1}^nf(\lambda_j)\mu_j^2$, where $\{\mu_j\}_{j=1}^n$ are the elements of the \textit{measure vector} $\boldsymbol{\mu} = \textbf{Q}^T\textbf{v} = [\mu_1,\ldots,\mu_n]^T$. With the construction of the measure $\mu(t)$
\begin{equation}
    \label{eq:measure_f}
        \mu(t)=\left\{
    \begin{aligned}
    &\quad 0 , & {\rm if} \ t < \lambda_1=a, \\
    &\sum_{j=1}^{k-1} \mu_j^2 , & {\rm if}\ \lambda_{k-1} \le t < \lambda_k, k=2,...,n, \\
    &\sum_{j=1}^{n} \mu_j^2 , & {\rm if} \ t \ge \lambda_n=b,
    \end{aligned}
    \right.
\end{equation}
one may consider the sum as a Riemann-Stieltjes integral $\mathcal{I}$
\begin{equation}
    \begin{aligned}
    \label{eq:RS}
    \mathcal{I}\stackrel{\mathrm{def}}{=}\int_{\lambda_1}^{\lambda_n}f(t) d\mu(t) =\sum_{j=1}^n f(\lambda_j)\mu_j^2.
\end{aligned}
\end{equation}
Further, according to the Gauss quadrature rule \cite[Chapter 6.2]{GM09}, the Riemann-Stieltjes integral $\mathcal{I}$ can be approximated by an $m$-point Lanczos quadrature rule $\mathcal{I}_{m}$ such that the last term in equation \eqref{eq:RS} reads
\begin{equation}
\label{eq:GQR}
\mathcal{I}_m \stackrel{\mathrm{def}}{=} \sum_{k=1}^{m} \tau_k f(\theta_k) \approx \mathcal{I},
\end{equation}
where $m$ is the number of quadrature nodes. Note that the quadrature nodes $\{\theta_k\}_{k=1}^m$ and the weights $\{\tau_k\}_{k=1}^m$ in \eqref{eq:GQR} can be obtained by the Lanczos algorithm \cite{BFG96,P98,GM09}. Let $\textbf{T}_{(m)}$ be the Jacobi matrix obtained in the Lanczos process. Then, the quadrature nodes are the eigenvalues of $\textbf{T}_{(m)}$, and the weights are the squares of the first elements of the corresponding normalized eigenvectors \cite{GW69}. Finally, one can obtain the approximation of $Q(\textbf{u},f,\textbf{A})$ as in \eqref{eq:GQ}, Algorithm \ref{alg:Lanc} details such a process\cite[Section 7.2]{GM09}.

%Now one should retake the scaling $\Arrowvert \textbf{u} \Arrowvert^2$ into calculation and the quadratic form $Q(\textbf{u},f,\textbf{A})$ is approximated by
%$$ Q_m(\textbf{u},f,\textbf{A}) \stackrel{\mathrm{def}}{=} \Arrowvert \textbf{u} \Arrowvert^2 \mathcal{I}_m \approx \Arrowvert \textbf{u} \Arrowvert^2 \mathcal{I} = Q(\textbf{u}, f, \textbf{A}).$$
%Algorithm \ref{alg:Lanc} outlines how to approximate the quadratic form $Q(\textbf{u},f,\textbf{A})$ via the Lanczos quadrature method \cite[Section 7.2]{GM09} without any re-orthogonalization technique. 

\begin{remark}
    To improve the numerical stability of Algorithm \ref{alg:Lanc}, one may complement several steps of reorthogonalization between lines 11 and 12.
\end{remark}
\begin{remark}
    If the Lanczos process breaks down before or at step $m$, Algorithm \ref{alg:Lanc} would compute \eqref{eq:GQ} exactly and terminates. 
\end{remark}

% {\color{red} \bf Remark 1}   \\ 
% {\color{red} \bf Remark 2} 
\begin{algorithm}[t]
        \raggedright
	\caption{Lanczos Quadrature Method for Quadratic Form Estimation} 
	\label{alg:Lanc}
	\hspace*{0.02in} {\bf Input:} Symmetric (positive definite) matrix $\textbf{A} \in \mathbb{R}^{n\times n}$, vector $\textbf{u} \in \mathbb{R}^{n}$, matrix function $f$, steps of Lanczos iterations $m$.\\
	\hspace*{0.02in} {\bf Output:} Approximation of the quadratic form $Q_m(\textbf{u}, f, \textbf{A}) \approx \textbf{u}^T f(\textbf{A}) \textbf{u}$.
	\begin{algorithmic}[1]
	    \State $\textbf{v}^{(1)} = \textbf{u}/\Arrowvert \textbf{u}\Arrowvert_2$
        \State $\alpha_1 = {\textbf{v}^{(1)}}^T \textbf{A} \textbf{v}^{(1)}$
	    \State $\textbf{u}^{(2)} = \textbf{A}\textbf{v}^{(1)} - \alpha_1 \textbf{v}^{(1)}$
		\For{$k = 2 \text{ to } m$}
			\State $\beta_{k-1} = \Arrowvert\textbf{u}^{(k)}\Arrowvert_2$
			\If{$\beta_{k-1} = 0$} 
                \State $m = k-1$
                \State \textbf{break}
                \EndIf 
                \State $\textbf{v}^{(k)} = \textbf{u}^{(k)}/\beta_{k-1}$
			\State $\alpha_{k} = {\textbf{v}^{(k)}}^T \textbf{A} \textbf{v}^{(k)}$
			\State $\textbf{u}^{(k+1)} = \textbf{A} \textbf{v}^{(k)} - \alpha_{k}             \textbf{v}^{(k)} - \beta_{k-1}\textbf{v}^{(k-1)}$
		\EndFor
            \State \textbf{end for}
            \State $\textbf{T}_{(m)} = \left[\begin{array}{ccccc}
               \alpha_1 &  \beta_1 & 0 & \cdots & 0 \\
               \beta_1 & \alpha_2 & \beta_2 & \ddots & \vdots \\
               0 & \beta_2 & \ddots & \ddots & 0 \\
               \vdots & \ddots & \ddots & \ddots & \beta_{m-1} \\
               0 & \cdots & 0 & \beta_{m-1} & \alpha_{m}
            \end{array}\right]$
            \State $[\mathbf{\Psi}_{(m)},\mathbf{\Theta}_{(m)}] = \texttt{eig}\left(\textbf{T}_{(m)}\right)$
            \State $[\tau_1,\ldots,\tau_m] = \left(\textbf{e}_1^T \mathbf{\Psi}_{(m)}\right)\odot\left(\textbf{e}_1^T \mathbf{\Psi}_{(m)}\right) $
            \State $[\theta_1,\ldots,\theta_m]^T = \texttt{diag}(\mathbf{\Theta}_{(m)})$
            \State \Return $Q_m(\textbf{u},f,\textbf{A}) = \Arrowvert\textbf{u}\Arrowvert_2^2 \sum_{k=1}^m \tau_k f(\theta_k)$
	\end{algorithmic}
\end{algorithm}

Error analyses for the approximation $\mathcal{I}_m$ were developed in \cite[Section 4.1]{UCS17} and \cite[Section 3]{CK21}. The summarized theorems are shown below for convenience.
\begin{theorem}
    \label{thm:UCSthm4.2} \cite[Theorem 4.2]{UCS17}
    Let $g$ be analytic in $\left[-1, 1\right]$ and analytically continuable in the open Bernstein ellipse $E_{\rho}$ with foci $\pm1$ and the sum of major and minor axis equal to $\rho > 1$. Let $M_\rho$ be the maximum of $|g(t)|$ on $E_\rho$. Then the $m$-point Lanczos quadrature approximation satisfies
    \begin{equation}
        \label{eq:UCSthm4.2}
        |\mathcal{I}-\mathcal{I}_m| \le \frac{4M_\rho}{1 - \rho^{-2}} \rho^{-2m}.
    \end{equation}
\end{theorem}
\begin{theorem}
    \label{thm:CK21thm3} \cite[Corollary 3]{CK21}
    Let $g$ be analytic in $\left[-1, 1\right]$ and analytically continuable in the open Bernstein ellipse $E_{\rho}$ with foci $\pm1$ and the sum of major and minor axis equal to $\rho > 1$. Let $M_\rho$ be the maximum of $|g(t)|$ on $E_\rho$. Then the $m$-point Lanczos quadrature approximation satisfies
    \begin{equation}
        \label{eq:CK21thm3}
        |\mathcal{I} - \mathcal{I}_m| \le \frac{4M_\rho}{1-\rho^{-1}} \rho^{-2m}.
    \end{equation}
\end{theorem}

To our knowledge, the discrepancy between Theorem \eqref{thm:UCSthm4.2} and Theorem \eqref{thm:CK21thm3} results from the symmetry of Lanczos quadrature. The authors of \cite{UCS17} derived the error bound \eqref{eq:UCSthm4.2} when the Gauss quadrature rule is symmetric, whereas \cite{CK21} remarked that the integral of the odd-degree Chebyshev polynomials against the measure $\mu(t)$ is actually not zero by symmetry. Thus, these two bounds \eqref{eq:UCSthm4.2} and \eqref{eq:CK21thm3} are valid when the quadrature rule is symmetric and asymmetric respectively. By fixing a tolerance $\epsilon$, the lower bounds of the required Lanczos iterations in the two cases are
\begin{equation}
    \label{eq:m_asym}
    m_{\text{asym}} \ge \frac{1}{2 \log(\rho)}\cdot \left[ \log(4M_{\rho}) - \log(1-\rho^{-1}) - \log(\epsilon) \right],
\end{equation}
\begin{equation}
        \label{eq:m_sym}
        m_{\text{sym}} \ge \frac{1}{2 \log(\rho)}\cdot \left[ \log(4M_{\rho}) - \log(1-\rho^{-2}) - \log(\epsilon) \right].
\end{equation}

Since $\rho>1$, the two bounds differ from the term
$$ m^{*} = m_{asym}-m_{sym}=\log(1+\rho^{-1})/2\log(\rho)>0.$$
Note that the elliptical radius is not fixed but is normally defined by $\rho = (\sqrt{\kappa} + 1)/(\sqrt{\kappa}- 1)$ \cite{UCS17}, where $\kappa$ is the condition number of matrix. Thus, the difference between \eqref{eq:m_asym} and \eqref{eq:m_sym} becomes greater when matrix is ill-conditioned. Moreover, since the choice of $m$ determines the computational complexity $\mathcal{O}(n^2m)$ of Algorithm \ref{alg:Lanc}, it is important to study the conditions under which the Lanczos process can generate symmetric Lanczos quadrature. 

On the other hand, as we can see from Algorithm \ref{alg:Lanc}, the quadrature nodes and weights are problem-dependent in approximating quadratic forms of a matrix function; they depend on both $\textbf{A}$ and $\textbf{u}$. For a stochastic trace estimator \cite{UCS17,CH23,E23,M21,P22}, $\textbf{u}$ is randomly distributed \cite{AT11}. In this scenario, it is non-trivial to provide a sufficient and necessary condition for a symmetric Lanczos quadrature. It is also interesting and of practical value to study how to choose a random $\textbf{u}$ for a certain class of matrices such that the sufficient condition can be realized.

\section{Symmetric Lanczos quadrature rule}
\label{sec:EA}
We first introduce Definition \ref{Def:sav} to characterize one of the two properties that a symmetric $m$-node Lanczos quadrature should have.
\begin{definition}
    \label{Def:sav}
    A vector $\textbf{w} \in \mathbb{R}^n$ is said to be an \textit{$r$-partial absolute palindrome} with an even number $r\le n$ if its elements satisfy
    \begin{equation}
        \label{eq:sav}
        |w_i| = |w_{n+1-i}|, i = 1,\ldots, \frac{r}{2}.
    \end{equation}
    If $|w_i| = |w_{n+1-i}|, i = 1,\ldots, \lfloor\frac{n}{2}\rfloor$, then $\textbf{w}$ is an \textit{absolute palindrome}.
\end{definition}
\begin{remark}
    Such a definition can help describe the symmetric equivalence of the Lanczos quadrature weights. Due to the parity of \eqref{eq:sav}, $r$ can only be an even number. Moreover, if a vector $\textbf{w} \in \mathbb{R}^{n}$ is an \textit{$r$-partial absolute palindrome}, then with the aid of permutation matrix $\textbf{P}_{(r/2)} = [\textbf{e}_{r/2},\textbf{e}_{r/2-1},\ldots,\textbf{e}_1]$ that reorders matrices in reverse order, the vector $|\textbf{w}| \in \mathbb{R}^{n}$ with absolute values can be written as
{
\renewcommand{\arraystretch}{2}
    \begin{equation}
        \label{wabs}
        |\textbf{w}|= \begin{bmatrix}
        |\textbf{w}|_{(r/2)} \\
        |\textbf{w}|_{(n-r)} \\
        \textbf{P}_{(r/2)} |\textbf{w}|_{(r/2)}
    \end{bmatrix},
\end{equation}
where the subscript represents the length of vector.
}

\end{remark}
\subsection{Necessary condition for symmetric Ritz values in the Lanczos process}
\begin{theorem}
\label{thm:nec}
    Let $\textnormal{\textbf{A}} = \textnormal{\textbf{Q}}\mathbf{\Lambda} \textnormal{\textbf{Q}}^T \in \mathbb{R}^{n \times n}$ be symmetric with rank $r$ and have $r$ distinct nonzero eigenvalues, let $\mathbf{\Lambda}$ have nondecreasing diagonal entries with $n-r$ zeros, and let $\textnormal{\textbf{v}}^{(1)}$ be a normalized unit vector for Lanczos iteration. If the $m$-node Lanczos quadrature is symmetric with respect to $0$ for all possible iterations in exact arithmetic ($m\le m^*$, where $m^*$ is the maximum iteration), then $\textnormal{\textbf{A}}$ has a symmetric eigenvalue distribution about $0$ and $\boldsymbol{\mu}^{(1)} = \textnormal{\textbf{Q}}^T \textnormal{\textbf{v}}^{(1)}$ is an $r$-partial absolute palindrome. 
\end{theorem}
\begin{proof}
    This proof is conducted in exact arithmetic. Let $\textbf{v}^{(1)}=\textbf{Q}\boldsymbol{\xi}$ where $\boldsymbol{\xi}$ is the coordinate normalized vector. For possible $m$ we obtain the same Jacobi matrices $\textbf{T}_{(m)}$ by applying the $m$-step Lanczos algorithm to either the diagonal matrix $\mathbf{\Lambda}$ and $\boldsymbol{\xi}$ or $\textbf{A}$ and $\textbf{v}^{(1)}$. Then our problem is equivalent to discussing the symmetry of an $m$-node Lanczos quadrature, associated with $\mathbf{\Lambda}$ and $\boldsymbol{\xi}$.

    Specifically, the Lanczos quadrature are required to be symmetric about $0$ in all possible iterations, so we consider the case of maximum iteration $m^*$. The $m^*$-node Lanczos quadrature $\mathcal{I}_{m^*}$ exactly computes the Riemann-Stieltjes integral $\mathcal{I}$ \eqref{eq:RS}, where the quadrature nodes form a subset of $\mathbf{\Lambda}$'s spectrum (or they are exactly the same when $m^* = n$, which means the $n$ eigenvalues of $\textbf{A}$ are assumed simple and $\textbf{v}^{(1)}$ is not deficient in any eigenvector of $\textbf{A}$). Meanwhile, if we assume that the quadrature weights with respect to symmetric nodes are identical in all iterations, the measure $\mu(t)$ \eqref{eq:measure_f} should be centro symmetric around $(0,0.5)$, where $0$ denotes the center of the spectrum of $\mathbf{\Lambda}$ and $0.5$ is one half of the summed weights. Recall that the increments at $t=\lambda_i, i = 1,\ldots,n$ of the piecewise constant function $\mu(t)$ are the squared entries of the measure vector $\boldsymbol{\mu}^{(1)} = \textbf{Q}^T\textbf{Q}\boldsymbol{\xi} = \boldsymbol{\xi}$. Then as $r \le m^*$, one must have
    $$ \xi_i^2 = \xi_{n+1-i}^2, \quad i = 1,\ldots,r/2. $$
    Thereby $\mathbf{\Lambda}$ has a symmetric eigenvalue distribution around $0$ and $\boldsymbol{\xi}$ is an $r$-partial absolute palindrome. 
    \qed
\end{proof}
From a practical perspective, it is difficult to obtain the spectral information of matrix in advance, e.g., the multiplicity of different eigenvalues. Thus, to avoid complexity and confusion, the necessary condition (Theorem \ref{thm:nec}) is valid for matrix $\textbf{A}$ with $r$ distinct nonzero eigenvalues. In general, the multiplicity for certain eigenvalues can be greater than $1$. In this case, even when the Lanczos quadrature is symmetric about $0$ for all iterations, not every pair of $\mu_i^2$ with respect to the symmetric eigenvalues is equal. Instead, the summation of the squares of the corresponding elements in $\boldsymbol{\mu}^{(1)}$ with respect to the symmetric eigenvalues should be numerically the same. For example, if the $m$-node Lanczos quadrature is always symmetric during iterations and $$ \lambda_i= \ldots = \lambda_{i+j} = -\lambda_{n+1-i-j} = \ldots = -\lambda_{n+1-i}, \quad j\ge 0,$$
then
$$\sum_{k=i}^{i+j}\mu_{k}^2 = \sum_{k=n+1-i-j}^{n+1-i}\mu_k^2.$$

\subsection{Sufficient condition for symmetric Ritz values in the Lanczos iterations}
\label{Sec:existenceSQN}

Next, we propose and prove a sufficient condition for symmetric Ritz values in the Lanczos iterations. Note that Wülling gave an example of symmetric Ritz values generated by the Lanczos iteration when $\textbf{A}$ is a diagonal matrix with exactly one zero eigenvalue with a symmetric spectrum around zero \cite{W05}. In that case, the author restricted the form of starting vectors as
$$ \textbf{v}^{(1)} = \left[\textbf{v}_{((n-1)/2)}; v_{(n+1)/2};\pm\textbf{v}_{((n-1)/2)}\right] $$
and assumed an odd dimension $n$ of $\textbf{A}$, and the example was left unproven. Instead, our proposed sufficient condition does not make any assumption on the parity of dimension.

\begin{lemma}
\label{Lem:Suf}
     Let $\textnormal{\textbf{A}} = \textnormal{\textbf{Q}}\mathbf{\Lambda} \textnormal{\textbf{Q}}^T \in \mathbb{R}^{n \times n}$ be symmetric with rank $r$, let $\mathbf{\Lambda} = \mathrm{diag}(\lambda_1,\lambda_2, \cdots,\lambda_n)$ with eigenvalues $\lambda_1\le \lambda_2 \le \ldots \le \lambda_n$ be symmetric about $0$, and let $\textnormal{\textbf{v}}^{(1)} \in \mathbb{R}^{n}$ be a normalized initial vector for the Lanczos iteration. For the $m$-step Lanczos method with $\textnormal{\textbf{A}}$ and $\textnormal{\textbf{v}}^{(1)}$ ($m\le m^*$, where $m^*$ is the maximum iteration), if vector $\boldsymbol{\mu}^{(1)} = \textnormal{\textbf{Q}}^T \textnormal{\textbf{v}}^{(1)}$ is an $r$-partial absolute palindrome, the Jacobi matrix $\textnormal{\textbf{T}}_{(m)}$ generated by $m$-step Lanczos iteration will have a constant zero diagonal.
\end{lemma}
\begin{proof}
    We divide the proof into two steps: first, we prove that the lemma holds when the matrix is of full rank, i.e., $r = n$, and then we extend it to the case where $r < n$. 

    Let $\textbf{P}$ denote the permutation matrix that reverses the order of entries, i.e., $\left(\textbf{P} \boldsymbol{\mu}^{(1)}\right)_i = \left(\boldsymbol{\mu}^{(1)}\right)_{n+1-i}, i = 1,\ldots,n$, and let $\textbf{S}$ be the signature matrix (with diagonal entries $\pm 1$) that ensures $\textbf{P}\boldsymbol{\mu}^{(1)} = \textbf{S} \boldsymbol{\mu}^{(1)}$. Note that $\textbf{P}$ and $\textbf{S}$ guarantee
    $$ \textbf{P}\mathbf{\Lambda} \textbf{P}^T = -\mathbf{\Lambda},\quad \mathbf{\Lambda} \textbf{S} = \textbf{S} \mathbf{\Lambda},\quad \textbf{S} = \textbf{S}^T,\quad \textbf{S}^2 = \textbf{I}. $$
    Denoting $\boldsymbol{\mu}^{(k)} = \textbf{Q}^T \textbf{v}^{(k)}$, we wish to prove
    \begin{equation}
        \label{eq:induct}
        \alpha_k = (\textbf{v}^{(k)})^T \textbf{A} \textbf{v}^{(k)} =  0, \quad \textbf{P}\boldsymbol{\mu}^{(k+1)} = (-1)^k \textbf{S} \boldsymbol{\mu}^{(k+1)},\quad k = 1,\ldots, m
    \end{equation}
    by mathematical induction. The second equality indicates that $\boldsymbol{\mu}^{(k+1)}$ is an absolute palindrome.

    \textbf{Base case}: For $k = 1$,
    $$ \begin{aligned}
        \alpha_1 &= {\textbf{v}^{(1)}}^T \textbf{Q} \mathbf{\Lambda} \textbf{Q}^T \textbf{v}^{(1)} =  {\boldsymbol{\mu}^{(1)}}^T \mathbf{\Lambda} \boldsymbol{\mu}^{(1)} = (\textbf{P}\boldsymbol{\mu}^{(1)})^T \textbf{P} \mathbf{\Lambda} \textbf{P}^T (\textbf{P} \boldsymbol{\mu}^{(1)}) \\
        &= (\textbf{S} \boldsymbol{\mu}^{(1)})^T (-\mathbf{\Lambda})(\textbf{S} \boldsymbol{\mu}^{(1)}) 
        = -{\boldsymbol{\mu}^{(1)}}^T \textbf{S}^T \mathbf{\Lambda} \textbf{S}\boldsymbol{\mu}^{(1)} = -{\boldsymbol{\mu}^{(1)}}^T \textbf{S}^T \textbf{S} \mathbf{\Lambda} \boldsymbol{\mu}^{(1)} \\
        &= -{\boldsymbol{\mu}^{(1)}}^T \mathbf{\Lambda} \boldsymbol{\mu}^{(1)} = -\alpha_1.
    \end{aligned} $$
    Clearly, $\alpha_1 = 0$. With the relationship between the first and second Lanczos vectors, we have
    $$\beta_1 \textbf{v}^{(2)} = \textbf{A} \textbf{v}^{(1)} - \alpha_1 \textbf{v}^{(1)} = \textbf{A} \textbf{v}^{(1)},  $$
    and thus
    $$ \boldsymbol{\mu}^{(2)} = \textbf{Q}^T \textbf{v}^{(2)} = \frac{1}{\beta_1} \textbf{Q}^T \textbf{Q} \mathbf{\Lambda} \textbf{Q}^T \textbf{v}^{(1)} = \frac{1}{\beta_1} \mathbf{\Lambda} \boldsymbol{\mu}^{(1)}.$$
    Then
    $$ 
    \begin{aligned}
        \textbf{P} \boldsymbol{\mu}^{(2)} &= \frac{1}{\beta_1} \textbf{P} \mathbf{\Lambda} \textbf{P}^T \textbf{P} \boldsymbol{\mu}^{(1)} = - \frac{1}{\beta_1} \mathbf{\Lambda} \textbf{P} \boldsymbol{\mu}^{(1)} = -\frac{1}{\beta_1} \mathbf{\Lambda} \textbf{S}\boldsymbol{\mu}^{(1)} \\
        &= -\textbf{S} \left(\frac{1}{\beta_1} \mathbf{\Lambda} \boldsymbol{\mu}^{(1)}\right) = -\textbf{S} \boldsymbol{\mu}^{(2)},
    \end{aligned} $$
    which proves the second equation in \eqref{eq:induct}. 
    
    \textbf{Inductive steps}: Assume \eqref{eq:induct} is correct for $k < m$. We prove that \eqref{eq:induct} also holds for $k + 1$,
    $$ 
    \begin{aligned}
        \alpha_{k+1} &= {\boldsymbol{\mu}^{(k+1)}}^T \mathbf{\Lambda} \boldsymbol{\mu}^{(k+1)} = (\textbf{P} \boldsymbol{\mu}^{(k+1)})^T\textbf{P} \mathbf{\Lambda} \textbf{P}^T \textbf{P} \boldsymbol{\mu}^{(k+1)} \\
        &= -(\textbf{S}\boldsymbol{\mu}^{(k+1)})^T \mathbf{\Lambda} \textbf{S}\boldsymbol{\mu}^{(k+1)} = -\alpha_{k+1},
    \end{aligned}$$
    which gives $\alpha_{k+1} = 0$. Based on the three-term recurrence between Lanczos vectors 
    $$ \beta_{k+1} \textbf{v}^{(k+2)} = \textbf{A} \textbf{v}^{(k+1)} - \alpha_{k+1}\textbf{v}^{(k+1)} - \beta_k \textbf{v}^{(k)} = \textbf{A}\textbf{v}^{(k+1)} - \beta_k \textbf{v}^{(k)},$$ 
    it is trivial that
    $$ \begin{aligned}
        \boldsymbol{\mu}^{(k+2)} &= \textbf{Q}^T \textbf{v}^{(k+2)} = \frac{1}{\beta_{k+1}}\left( \textbf{Q}^T \textbf{Q} \mathbf{\Lambda} \textbf{Q}^T \textbf{v}^{(k+1)} - \beta_k \textbf{Q}^T \textbf{v}^{(k)} \right) \\ &= \frac{1}{\beta_{k+1}} \left(\mathbf{\Lambda} \boldsymbol{\mu}^{(k+1)} - \beta_k \boldsymbol{\mu}^{(k)}\right).
    \end{aligned}$$
    Then 
    $$ \begin{aligned}
        \textbf{P}\boldsymbol{\mu}^{(k+2)} &= \frac{1}{\beta_{k+1}}\left(\textbf{P}\mathbf{\Lambda} \textbf{P}^T \textbf{P} \boldsymbol{\mu}^{(k+1)} - \beta_k \textbf{P} \boldsymbol{\mu}^{(k)}\right)\\
        &= \frac{1}{\beta_{k+1}} \left( -\mathbf{\Lambda} \textbf{P} \boldsymbol{\mu}^{(k+1)} - \beta_k \textbf{P} \boldsymbol{\mu}^{(k)} \right) \\
        &= \frac{1}{\beta_{k+1}} \left( -(-1)^k \mathbf{\Lambda} \textbf{S} \boldsymbol{\mu}^{(k+1)} - (-1)^{k-1}\beta_k \textbf{S} \boldsymbol{\mu}^{(k)} \right) \\
        &= (-1)^{k+1} \textbf{S} \left(\frac{1}{\beta_{k+1}} \left( \mathbf{\Lambda} \boldsymbol{\mu}^{(k+1)} - \beta_k \boldsymbol{\mu}^{(k)} \right) \right) \\
        & = (-1)^{k+1} \textbf{S} \boldsymbol{\mu}^{(k+2)},
    \end{aligned} $$
    which completes the proof of \eqref{eq:induct} in case when $\textbf{A}$ has full rank. 

    For rank-deficient matrices that have $r$ nonzero eigenvalues, on the basis of the $r$-partial absolute palindrome $\boldsymbol{\mu}$, one may find the required $\textbf{P}$ and $\textbf{S}$ that satisfy $\left(\textbf{P}\boldsymbol{\mu}^{(1)}\right)_i = \left(\textbf{S}\boldsymbol{\mu}^{(1)}\right)_i, i = 1,\ldots, r/2,n+1-r/2,\ldots,n$. The $n-r$ values in the middle of the $\boldsymbol{\mu}^{(1)}$ vector do not affect the computation of the quadratic form, as the corresponding elements in the diagonal matrix $\mathbf{\Lambda}$ are all zero. In this case, the proposition to be proven by mathematical induction becomes
    \begin{equation}
        \alpha_k = (\textbf{v}^{(k)})^T \textbf{A} \textbf{v}^{(k)} =  0, \quad \left(\textbf{P}\boldsymbol{\mu}^{(k+1)}\right)_i = (-1)^k \left(\textbf{S} \boldsymbol{\mu}^{(k+1)}\right)_i, 
    \end{equation}
    where $k = 1,\ldots,m$ and $i = 1,\ldots, r/2,n+1-r/2,\ldots,n$. One may follow the same procedure as shown in the $r = n$ case to prove Lemma \ref{Lem:Suf}.
    \qed
\end{proof}

\begin{remark}
   One may also prove Theorem \ref{Thm:main} based on the theory of orthogonal polynomials. Let $\{l_{k}\}_{k=1}^m$ denote the orthogonal polynomials generated by the Lanczos three-term recurrence
   $$ \beta_{k}l_{k+1}(x) = (x - \alpha_k)l_k(x) - \beta_{k-1}l_{k-1}(x),
   $$
   which work on an interval symmetric about $0$. Then the roots of $l_k$ are symmetric about $0$ for $k = 0,1,\ldots$ if and only if the orthogonal polynomials of even degree are even and those of odd degree are odd functions. Thus $\alpha_k = 0$ must hold.
   
   In addition, since we use mathematical induction to prove \eqref{eq:induct} holds for every single step $k$, the conclusion of Lemma \ref{Lem:Suf} always holds even when the orthogonality is lost in finite precision arithmetic. Moreover, for a real symmetric matrix $\textbf{A}$ with rank $r \le n$, any $r^*$-partial absolute palindrome $\boldsymbol{\mu}^{(1)}$ with $r^* \ge r$ helps generate constant diagonal entries $\bar{\lambda}$ during the Lanczos process. Thus, without rank information, it is wise to let $\boldsymbol{\mu}^{(1)}$ be an absolute palindrome.
\end{remark}

Under the same assumption of Lemma \ref{Lem:Suf}, we prove that the Lanczos process would result in symmetrically distributed Ritz values and that the quadrature weights with respect to the pairs of symmetric quadrature nodes are equal. Before diving into details, it is necessary to first note that the Jacobi matrices of $m$ zero diagonals can be reordered to Jordan-Wielandt matrices \eqref{Eq:A}, namely the block form $\begin{bmatrix} & \textbf{B} \\ \textbf{B}^T & \end{bmatrix}$, where $\textbf{B} \in \mathbb{R}^{n_1\times n_2}$ and $m = n_1 + n_2$. The rearrangement follows the red-black ordering \cite[p. 211]{HY81} \cite[p. 123]{S03}, where either $n_1 = n_2$ ($m$ is even) or $n_1 = n_2 + 1$ ($m$ is odd). There are several works demonstrating that the Jordan-Wielandt matrices have symmetric eigenvalue distributions \cite[Theorem 1.2.2]{B96}\cite[Proposition 4.12]{S03}.
\begin{theorem}
\label{thm:m=n} \cite[Theorem 1.2.2]{B96}
    Let the singular value decomposition of $\textnormal{\textbf{B}} \in \mathbb{C}^{n_1 \times n_2}$ be $\textnormal{\textbf{B}} = \textnormal{\textbf{U}} \mathbf{\Sigma} \textnormal{\textbf{V}}^H,$ where $\mathbf{\Sigma} = \begin{bmatrix}
        \mathbf{\Sigma}_1 & \textnormal{\textbf{O}} \\ \textnormal{\textbf{O}} & \textnormal{\textbf{O}}
    \end{bmatrix}$, $\mathbf{\Sigma}_1 = \mathrm{diag}(\sigma_1, \sigma_2, \cdots, \sigma_{r_{\textbf{B}}})$, $r_{\textbf{B}} = \mathrm{rank}(\textnormal{\textbf{B}})$,
    $$\textnormal{\textbf{U}} = \left[\textnormal{\textbf{U}}_1, \textnormal{\textbf{U}}_2\right], \quad\textnormal{\textbf{U}}_1 \in \mathbb{C}^{n_1 \times r_{\textbf{B}}}, \quad\textnormal{\textbf{U}}_2 \in \mathbb{C}^{n_1 \times (n_1-r_{\textbf{B}})},$$
    $$\textnormal{\textbf{V}} = \left[\textnormal{\textbf{V}}_1, \textnormal{\textbf{V}}_2\right],\quad \textnormal{\textbf{V}}_1 \in \mathbb{C}^{n_2 \times r_{\textbf{B}}}, \quad\textnormal{\textbf{V}}_2 \in \mathbb{C}^{n_2 \times (n_2-r_{\textbf{B}})}.$$
    Then
    $$\textnormal{\textbf{A}} = \begin{bmatrix}
        \textnormal{\textbf{O}}_{(n_1)} & \textnormal{\textbf{B}} \\ \textnormal{\textbf{B}}^H & \textnormal{\textbf{O}}_{(n_2)}
    \end{bmatrix} = \textnormal{\textbf{Q}}^H \begin{bmatrix}
        \mathbf{\Sigma}_1 &  &  \\  &  -\mathbf{\Sigma}_1 & \\    &  & \textnormal{\textbf{O}} 
    \end{bmatrix} \textnormal{\textbf{Q}},$$
    where blank positions and $\textnormal{\textbf{O}}$ denote zero blocks of appropriate sizes, $\textnormal{\textbf{Q}}$ is unitary and is represented as
    $$ \textnormal{\textbf{Q}} = \frac{1}{\sqrt{2}}\begin{bmatrix}
        \textnormal{\textbf{U}}_1 & \textnormal{\textbf{U}}_1 & \sqrt{2} \textnormal{\textbf{U}}_2 & \textnormal{\textbf{O}} \\ \textnormal{\textbf{V}}_1 & -\textnormal{\textbf{V}}_1 & \textnormal{\textbf{O}} & \sqrt{2}\textnormal{\textbf{V}}_2 
    \end{bmatrix}^H.$$
    Thus $2r_{\textbf{B}}$ eigenpairs of $\textnormal{\textbf{A}}$ are $\left(\pm \sigma_i, \begin{bmatrix}
        \textnormal{\textbf{u}}_i \\
        \pm \textnormal{\textbf{v}}_i
    \end{bmatrix}\right), i = 1,\ldots,r_{\textbf{B}}$, and zero eigenvalues repeated $(n_1+n_2-2r_{\textbf{B}})$ times, where $\textnormal{\textbf{u}}_i$ and $\textnormal{\textbf{v}}_i$ are the columns of $\textnormal{\textbf{U}}_1$ and $\textnormal{\textbf{V}}_1$ respectively.
\end{theorem}
Based on this important property, it is trivial to derive
\begin{theorem}
\label{Thm:main}
    Let $\textnormal{\textbf{A}} = \textnormal{\textbf{Q}}\mathbf{\Lambda} \textnormal{\textbf{Q}}^T \in \mathbb{R}^{n \times n}$ be symmetric with rank $r$, let $\mathbf{\Lambda} = \mathrm{diag}(\lambda_1,\lambda_2,\cdots,\lambda_n)$ with eigenvalues $\lambda_1\le \lambda_2 \le \ldots \le \lambda_n$ be symmetric about $0$, and let $\textnormal{\textbf{v}}^{(1)} \in \mathbb{R}^{n}$ be a normalized initial vector for the Lanczos iteration. For the $m$-step Lanczos method with $\textnormal{\textbf{A}}$ and $\textnormal{\textbf{v}}^{(1)}$ ($m\le m^*$, where $m^*$ is the maximum iteration), if vector $\boldsymbol{\mu}^{(1)} = \textnormal{\textbf{Q}}^T \textnormal{\textbf{v}}^{(1)}$ is an $r$-partial absolute palindrome, the distribution of $m$ Ritz values will be symmetric about $0$, and the quadrature weights corresponding to the pairs of symmetric quadrature nodes are equal.
\end{theorem}
\begin{proof}
    First recall that the Jacobi matrix $\textbf{T}_{(m)}$ generated by the $m$-step Lanczos iteration has $m$ simple eigenvalues if there is no breakdown. Then, based on Lemma \ref{Lem:Suf} and Theorem \ref{thm:m=n}, no matter $m$ is even or odd, $\textbf{T}_{(m)}$ has $n_2 = \lfloor\frac{m}{2}\rfloor$ symmetric pairs of eigenpairs 
    $$\Bigg\{\left(\pm\sigma_i, \begin{bmatrix}
        \textbf{u}_i \\
        \pm\textbf{v}_i
    \end{bmatrix}\right)\Bigg\}_{i=1}^{n_2}.$$
    Specifically, for odd $m$, an additional zero eigenvalue lies on the spectrum of $\textbf{T}_{(m)}$. Thus, the $m$ Ritz values exhibit symmetry about $0$. Furthermore, based on the equivalence of the pairs of first entries in the eigenvectors corresponding to the symmetric eigenvalues, symmetrically equal quadrature weights are guaranteed for such Gaussian quadrature rules.
    \qed
\end{proof}

\begin{remark}
    Note that both Lemma \ref{Lem:Suf} and Theorem \ref{Thm:main} discuss matrices with symmetric spectrum with respect to $0$. Now suppose $\textnormal{\textbf{A}} \to \textnormal{\textbf{A}} - s\textnormal{\textbf{I}}$ and $r$ means the number of eigenvalues satisfying $\lambda_j\neq s, j = 1,\ldots,n$. Then an $r$-partial absolute palindrome would guarantee a \textbf{shifted} symmetric quadrature (nodes are symmetric about the shift $s$, and weights corresponding to the pairs of symmetric nodes possess equivalence) with respect to the shifted matrix due to the `Lanczos invariance'. That is, if we apply the Lanczos method on $\textnormal{\textbf{A}} - s\textnormal{\textbf{I}}$ instead of $\textnormal{\textbf{A}}$ but with the same initial vector, then we would get the same Lanczos vectors and the Jacobi matrix $\textnormal{\textbf{T}}_{(m)}$ transforms as $\textnormal{\textbf{T}}_{(m)} \to \textnormal{\textbf{T}}_{(m)} - s\textnormal{\textbf{I}}_{(m)}$.
\end{remark}

Next, we aim at demonstrating that Theorem \ref{Thm:main} is not only theoretical but also practical. Recall that Jordan-Wielandt matrices \eqref{Eq:A} exist in various applications and perfectly match the requirement of symmetric matrices and the symmetry of corresponding eigenvalues. For example, the graph of a finite difference matrix is bipartite, meaning that the vertices can be divided into two sets in red-black order so no edges exist in each set \cite[p. 211]{HY81} \cite[p. 123]{S03}. In the analysis of complex networks, a directed network of $n$ nodes with an asymmetric adjacency matrix $\textbf{B} \in \mathbb{R}^{n\times n}$ can be extended to a bipartite undirected network with a symmetric \emph{block supra-adjacency matrix} of form \eqref{Eq:A} via bipartization \cite{BB20,BEK13,BS22}. A subsequent inquiry emerges: Does there exist an $r$-partial absolute palindrome $\boldsymbol{\mu}^{(1)}$ associated with Jordan-Wielandt matrices? If such a construct exists, what properties must the initial vectors (for the Lanczos iteration) possess to guarantee its formation?

\subsection{Realization of the sufficient condition on the Jordan-Wielandt matrices}

\label{sec:construction}
In this section, we discuss how to find a simple type of initial vectors to guarantee $r$-partial absolute palindrome that guarantee a symmetric Lanczos quadrature.

\label{sec:case}
\begin{theorem}
    \label{prop:2}
    Let $\textnormal{\textbf{B}} \in \mathbb{R}^{n_1 \times n_2}$ and $\textnormal{\textbf{A}} = \textnormal{\textbf{Q}}\mathbf{\Lambda} \textnormal{\textbf{Q}}^T = \begin{bmatrix}
         & \textnormal{\textbf{B}} \\
        \textnormal{\textbf{B}}^T & 
    \end{bmatrix} \in \mathbb{R}^{(n_1+n_2)\times (n_1+n_2)}$ with rank $r$ and nondecreasing values on the diagonal of $\mathbf{\Lambda}$. Then, any initial vector that has either form
    $$ \textnormal{\textbf{v}}^{(1)} = \begin{bmatrix}
        \textnormal{\textbf{v}}_u \\
        \textnormal{\textbf{0}}_{(n_2)}
    \end{bmatrix}  \quad \textit{or} \quad \textnormal{\textbf{v}}^{(1)} = \begin{bmatrix}
        \textnormal{\textbf{0}}_{(n_1)} \\
        \textnormal{\textbf{v}}_d
    \end{bmatrix}$$
    with real vector $\textnormal{\textbf{v}}_u \in \mathbb{R}^{n_1}$, $\textnormal{\textbf{v}}_d \in \mathbb{R}^{n_2}$ and zero vectors $\textnormal{\textbf{0}}_{(n_1)}\in\mathbb{R}^{n_1}, \textnormal{\textbf{0}}_{(n_2)}\in\mathbb{R}^{n_2}$ guarantees an $r$-partial absolute palindrome $\boldsymbol{\mu}^{(1)} = \textnormal{\textbf{Q}}^T \textnormal{\textbf{v}}^{(1)}$.
\end{theorem}

\begin{proof}
    Based on the block form of $\textbf{A}$, it is trivial that $r = 2r_{\textbf{B}}$, where $r_{\textbf{B}}$ is the rank of $\textbf{B}$. Based on suitable permutation, $\textbf{A}$ can be factorized as $\textbf{A} = \textbf{Q} \mathbf{\Lambda}\textbf{Q}^T$, where the block forms of $\mathbf{\Lambda}$ and $\textbf{Q}^T$ read

    \vspace{0.5cm}

    \NiceMatrixOptions{columns-width=12mm}

    \begin{equation*}
        \mathbf{\Lambda} = \begin{bNiceArray}{c|c|c}[margin,last-col]
\mathbf{\Lambda}_{11} &  &  & \quad r_{\textbf{B}} \\ \Hline
 &  \Block{1-1}{} &  & \quad n_1+n_2-r \\ \Hline
 & & \mathbf{\Lambda}_{33} &  \quad r_{\textbf{B}}
\CodeAfter
 \OverBrace[yshift=1.5mm]{1-1}{1-1}{r_{\textbf{B}}}
 \OverBrace[yshift=1.5mm]{1-2}{1-2}{n_1+n_2-r}
 \OverBrace[yshift=1.5mm]{1-3}{1-3}{r_{\textbf{B}}}
 \SubMatrix{.}{1-1}{1-3}{\rbrace}[xshift=5mm]
 \SubMatrix{.}{2-1}{2-3}{\rbrace}[xshift=5mm]
 \SubMatrix{.}{3-1}{3-3}{\rbrace}[xshift=5mm]
\end{bNiceArray}
    \end{equation*}
\vspace{0.4cm}\
    \NiceMatrixOptions{columns-width=10mm}
    \NiceMatrixOptions{cell-space-limits = 6pt}
    \begin{equation*}
        \textbf{Q}^T = \frac{1}{\sqrt{2}}\begin{bNiceArray}{c|c}[margin,last-col]
\textbf{Q}_{11}^T & \textbf{Q}_{21}^T & \quad r_{\textbf{B}} \\ \Hline
 \textbf{Q}_{12}^T & \textbf{Q}_{22}^T  & \quad n_1+n_2-r \\ \Hline
 \textbf{Q}_{13}^T & \textbf{Q}_{23}^T  & \quad r_{\textbf{B}} \\ 
\CodeAfter
 \OverBrace[yshift=1.5mm]{1-1}{1-1}{n_1}
 \OverBrace[yshift=1.5mm]{1-2}{1-2}{n_2}
 \SubMatrix{.}{1-1}{1-2}{\rbrace}[xshift=5mm]
 \SubMatrix{.}{2-1}{2-2}{\rbrace}[xshift=5mm]
 \SubMatrix{.}{3-1}{3-2}{\rbrace}[xshift=5mm]
\end{bNiceArray}
    \end{equation*}
    Note that the $3\times3$ block form of $\mathbf{\Lambda}$ is reduced to $2\times 2$ and $3\times 2$ block of $\textbf{Q}^T$ turns $2\times 2$ if $r = n_1 + n_2$. For nondecreasing diagonal $\mathbf{\Lambda}$, one may set $\mathbf{\Lambda}_{33} = -\textbf{P}_{(r_{\textbf{B}})}\mathbf{\Lambda}_{11}$ without loss of generality, where $\textbf{P}_{(r_{\textbf{B}})}$ is the $r_{\textbf{B}}\times r_{\textbf{B}}$ anti-diagonal identity matrix. Then $\textbf{Q}_{13}^T=\textbf{P}_{(r_{\textbf{B}})}\textbf{Q}_{11}^T$ and $\textbf{Q}_{23}^T = -\textbf{P}_{(r_{\textbf{B}})} \textbf{Q}_{21}^{T}$. Denote $\textbf{v}^{(1)} = \begin{bmatrix}
    \textbf{v}_u \\ \textbf{v}_d
\end{bmatrix} \in \mathbb{R}^{n_1 + n_2}$ with $\textbf{v}_u \in \mathbb{R}^{n_1}, \textbf{v}_d \in \mathbb{R}^{n_2}$. Then $\boldsymbol{\mu}^{(1)}$ is represented as 
    $$
    \begin{aligned}
        \boldsymbol{\mu}^{(1)} = \textbf{Q}^T \textbf{v}^{(1)} & = \begin{bmatrix}
        \textbf{Q}_{11}^T & \quad \textbf{Q}_{21}^T \\
        \textbf{Q}_{12}^T & \quad \textbf{Q}_{22}^T \\
       \textbf{P}_{(r_{\textbf{B}})} \textbf{Q}_{11}^T & \quad-\textbf{P}_{(r_{\textbf{B}})} \textbf{Q}_{21}^T
    \end{bmatrix} \begin{bmatrix}
            \textbf{v}_u \\
            \textbf{v}_d
        \end{bmatrix} \\
        & =  \begin{bmatrix}
            \textbf{Q}_{11}^T \textbf{v}_u + \textbf{Q}_{21}^T \textbf{v}_d \\
            \textbf{Q}_{12}^T \textbf{v}_u + \textbf{Q}_{22}^T \textbf{v}_d \\
            \textbf{P}_{(r_{\textbf{B}})}(\textbf{Q}_{11}^T \textbf{v}_u - \textbf{Q}_{21}^T \textbf{v}_d)
        \end{bmatrix}.
    \end{aligned}
    $$
    One convenient and economic (by saving memory) choice is to take $\textbf{v}_u = \textbf{0}_{(n_1)} \in \mathbb{R}^{n_1}$ or $\textbf{v}_d = \textbf{0}_{(n_2)} \in \mathbb{R}^{n_2}$ to ensure an $r$-partial absolute palindrome $\boldsymbol{\mu}^{(1)}$.
    \qed
\end{proof}

\subsection{Computational advantages through bidiagonalization}

So far we have discussed sufficient conditions for symmetric Lanczos quadrature and how to guarantee the symmetry of quadrature for Jordan-Wielandt matrices. Now we show how to reduce the computational burden in the Lanczos process when $\textbf{B}$ in $\textbf{A} = \begin{bmatrix}
     & \textbf{B} \\ \textbf{B}^T & 
\end{bmatrix}$ is accessible. First, recall the Lanczos bidiagonalization algorithm (see Algorithm \ref{alg:LancBidi}) that reduces a given matrix $\textbf{B}\in\mathbb{R}^{n_1\times n_2}$ to a bidiagonal matrix $\textbf{J}_{(m)},\in\mathbb{R}^{m \times m}$, $m\le \min\{n_1,n_2\}$. Two sequences of orthonormal vectors $\{\textbf{u}^{(k)}\}_{k=1}^{m}$ and $\{\textbf{v}^{(k)}\}_{k=1}^{m}$ are generated in the process, with their norms stored in the diagonals and off-diagonals of $\textbf{J}_{(m)}$ respectively. We observe that the sequence of interlaced elements $\{\alpha_1,\beta_1,\alpha_2,\beta_2,\ldots,\beta_{m-1},\alpha_{m}\}$ in $\textbf{J}_{(m)}$ generated by the $m$-step Lanczos bidiagonalization with $\textbf{B}$ (resp. $\textbf{B}^T$) and $\textbf{v}$ are equivalent to the off-diagonals of $\textbf{T}_{(2m)}$ generated by the $2m$-step Lanczos algorithm with $\textbf{A}$ and $[\textbf{v}_{(n_1)}; \textbf{0}]$ (resp. $[\textbf{0};\textbf{v}_{(n_2)}]$). Meanwhile, thanks to Lemma \ref{Lem:Suf} and Theorem \ref{prop:2}, the diagonal entries of $\textbf{T}_{(2m)}$ are all zeros. Then 
$$\begin{bmatrix}
     & \textbf{J}_{(m)}\\
     \textbf{J}_{(m)}^T & \end{bmatrix}$$
can be transformed to $\textbf{T}_{(2m)}$ through proper permutations. Golub and Kahan have discussed the relationship between the eigenpairs $\{\theta_k,\boldsymbol{\psi}_k\}_{k=1}^{2m}$ of $\textbf{T}_{(2m)}$ and the singular triplets $\{\delta_k,\textbf{x}_k,\textbf{y}_k\}_{k=1}^m$ of $\textbf{J}_{(m)}$ \cite[p.212-213]{G65}. Recall that $\{\theta_k\}_{k=1}^{2m} = \{\pm\delta_k\}_{k=1}^m$ are the eigenvalues of $\textbf{T}_{(2m)}$, due to Theorem \ref{thm:m=n}. Then, letting $\psi_{2i} = x_i$ and $\psi_{2i-1} = \pm y_i$, we have
$$ \textbf{T}_{(2m)}\boldsymbol{\psi}=\theta\boldsymbol{\psi} = \pm \delta\boldsymbol{\psi}.$$
Such an equivalence demonstrates that when $\textbf{A}$ is a Jordan-Wielandt matrix and one wishes to guarantee a symmetric Lanczos quadrature, he/she can directly conduct the Lanczos bidiagonalization on $\textbf{B}$ to save memory and reduce time on solving the half-scale singular value problem. In addition, in the calculation of the $2m$-node Lanczos quadrature, the quadrature nodes are the pairs of $\pm$ singular values of $\textbf{J}_{(m)}$ and the quadrature weights $[\tau_1,\ldots,\tau_m] = \frac{1}{2}\left(\textbf{e}_1^T \textbf{Y}_{(m)}\right)\odot\left(\textbf{e}_1^T \textbf{Y}_{(m)}\right)$, where $\textbf{Y}_{(m)}$ is the orthonormal matrix that stores the right singular vectors of $\textbf{J}_{(m)}$.

However, when explicit access to $\textbf{B}$ is not available, for instance, when $\textbf{A}$ is provided only as a black-box matrix-vector multiplication routine, one is constrained to use the standard Lanczos tridiagonalization on $\textbf{A}$.

\begin{algorithm}[t]
        \raggedright
	\caption{Lanczos Bidiagonalization Algorithm} 
	\label{alg:LancBidi}
	\hspace*{0.02in} {\bf Input:} Matrix $\textbf{B} \in \mathbb{R}^{n_1 \times n_2}$, vector $\textbf{v} \in \mathbb{R}^{n_2}$, number of iterations $m$.\\
	\hspace*{0.02in} {\bf Output:} Bidiagonal matrix $\textbf{J}_{(m)} \in \mathbb{R}^{m \times m}$.
	\begin{algorithmic}[1]
	    \State $\textbf{v}^{(1)} = \textbf{v}/\Vert \textbf{v} \Vert_2$ 
        \State $\beta_0 = 0$ 
        \State $\textbf{u}^{(0)} = \mathbf{0}$ 
	    \For{$k = 1$ \textbf{to} $m$}
			\State $\textbf{u} = \textbf{B} \textbf{v}^{(k)} - \beta_{k-1} \textbf{u}^{(k-1)}$ 
			\State $\alpha_k = \Vert \textbf{u} \Vert_2$
			\If{$\alpha_k == 0$}
                \State $m = k-1$ 
                \State \textbf{break}
            \EndIf 
                \State $\textbf{u}^{(k)} = \textbf{u} / \alpha_k$
			\If{$k < m$} 
			    \State $\textbf{v} = \textbf{B}^T \textbf{u}^{(k)} - \alpha_k \textbf{v}^{(k)}$
			    \State $\beta_k = \Vert \textbf{v} \Vert_2$
			    \If{$\beta_k == 0$}
                    \State $m = k$ 
                    \State \textbf{break}
                \EndIf 
                \State $\textbf{v}^{(k+1)} = \textbf{v} / \beta_k$
			\EndIf
		\EndFor
        \State \textbf{end for}
        \State \Return $\textbf{J}_{(m)} = \begin{bmatrix}
        \alpha_1 & \beta_1 & 0 & \cdots & 0 \\
        0 & \alpha_2 & \beta_2 & \ddots & \vdots \\
        \vdots & \ddots & \ddots & \ddots & 0 \\
        0 & \cdots & 0 & \alpha_{m-1} & \beta_{m-1} \\
        0 & \cdots & 0 & 0 & \alpha_m
        \end{bmatrix}$
	\end{algorithmic}
\end{algorithm}

\section{Application: estimation of the Estrada index}
Given an adjacency matrix $\textbf{A} \in \mathbb{R}^{n \times n}$ with respect to a graph $G$ of order $n$ and a real parameter $\beta$, the Estrada index (EI)
\begin{equation}
    \label{eq:estrada}
    EI(\textbf{A},\beta) = \sum_{i=1}^n e^{\beta \lambda_i} = \mathrm{tr}(e^{\beta \textbf{A}})
\end{equation}
is an important indicator that measures the complexity, connectivity and robustness of networks in applications ranging from chemistry and molecular design, protein structure analysis, and complex network analysis to social science \cite{E00,BB10,E12,K14}. The Hutchinson trace estimator \cite{H90} estimates $\mathrm{tr}(f(\textbf{A}))$ by
\begin{equation}
    \label{eq:Hutchinson}
    \mathrm{tr}(f(\textbf{A})) = \sum_{i=1}^{n}[f(\textbf{A})]_{ii} \approx \frac{1}{N}\sum_{k=1}^N {\textbf{z}_k}^T f(\textbf{A}) \textbf{z}_k,
\end{equation}
where the entries of $\textbf{z}_k \in \mathbb{R}^n$ independently follow the Rademacher distribution, i.e., every element takes a value of $\pm 1$ with a probability of $1/2$ for each, and $\mathbb{E}[\textbf{z}_k\textbf{z}_k^T] = \textbf{I}$. The Gaussian trace estimator or normalized Rayleigh-quotient trace estimator also helps approximate the trace \cite{AT11}. The quadratic forms in \eqref{eq:Hutchinson} can be approximated by the Gauss quadrature \cite{UCS17} as described by Algorithm \ref{alg:Lanc}. 

For general simple graphs, no additional spectral information of their adjacency matrices is given, so Theorem \ref{Thm:main} is normally not applicable. Nevertheless, the property of a symmetric eigenvalue distribution can be captured for adjacency matrices in two cases.
\begin{itemize}
    \item In the analysis of layer-coupled multiplex networks, a directed network of $n$ nodes with an asymmetric adjacency matrix $\textbf{B} \in \mathbb{R}^{n_1\times n_1}$ can be extended to a bipartite undirected network with a symmetric \textit{block supra-adjacency matrix} via bipartization \cite{BB20,BEK13,BS22}.
    \item Graphs are undirected and bipartite \cite{Z09}. The numbers of nodes in the two sets are not necessarily equal. The full adjacency matrices of these types of graphs are also of the form \eqref{Eq:A}, where $\textbf{B} \in \mathbb{R}^{n_1 \times n_2}$.
\end{itemize}
For adjacency matrices of the form \eqref{Eq:A}, even if the vectors $\{\textbf{z}_k\}_{k=1}^N$ in the quadratic form (i.e., the initial vector of the Lanczos process) are chosen randomly, we can still guarantee, as discussed in Section \ref{sec:construction}, that the distribution of Ritz values remains symmetric at any iteration step $m \le m^*$. Based on Theorem \ref{prop:2}, we suggest the use of initial vectors with Rademacher entries ($\pm 1$) and zeros (represented by \textit{partial Rademacher} vectors in this context) rather than a complete Rademacher distributed vector to guarantee the symmetry of quadrature rules, thus increasing the convergence rate.

Furthermore, we care about the unbiasedness of such a trace estimator with a partial Rademacher vector. The discussion is carried out in two cases: $\textbf{B} \in \mathbb{R}^{n_1 \times n_2}$ is square ($n_1 = n_2$), and $\textbf{B}$ is non-square ($n_1 \neq n_2$) according to practical applications.

\subsection{Case 1: $n_1 = n_2$}

When $\textbf{B} = \textbf{U}\mathbf{\Sigma} \textbf{V}^T \in \mathbb{R}^{n_1\times n_1}$ is a square matrix, $f(\textbf{A})$ is decomposed as
    \NiceMatrixOptions{cell-space-limits = 6pt}
$$
\begin{aligned}
    f(\textbf{A}) &= \frac{1}{2} \begin{bmatrix}
    \textbf{U} & \textbf{U} \\
    \textbf{V} & -\textbf{V}
\end{bmatrix} \begin{bmatrix}
    f(\mathbf{\Sigma}) &  \\
     & f(-\mathbf{\Sigma})
\end{bmatrix} \begin{bmatrix}
    \textbf{U}^T & \textbf{V}^T \\
    \textbf{U}^T & -\textbf{V}^T
\end{bmatrix} \\
    &= \frac{1}{2} \begin{bNiceArray}{cc}[margin]
        \textbf{U}\left( f(\mathbf{\Sigma}) + f(-\mathbf{\Sigma})\right)\textbf{U}^T & \quad \textbf{U}\left(f(\mathbf{\Sigma}) - f(-\mathbf{\Sigma})\right)\textbf{V}^T \\
        \textbf{V}\left(f(\mathbf{\Sigma}) - f(-\mathbf{\Sigma})\right)\textbf{U}^T & \quad \textbf{V}(f(\mathbf{\Sigma}) + f(-\mathbf{\Sigma}))\textbf{V}^T
    \end{bNiceArray}
\end{aligned}
$$
and the trace is represented as
$$\begin{aligned}
    \mathrm{tr}(f(\textbf{A})) &= 2\cdot \frac{1}{2}\cdot\mathrm{tr}\left(\textbf{U}\left(f(\mathbf{\Sigma}) + f(-\mathbf{\Sigma})\right)\textbf{U}^T\right)\\ &= \mathbb{E}\left[\textbf{z}_{(n_1)}^T \textbf{U}\left(f(\mathbf{\Sigma}) + f(-\mathbf{\Sigma})\right)\textbf{U}^T\textbf{z}_{(n_1)} \right].
\end{aligned}$$
The first equation stems from the cyclic property of the trace operator, which ensures that the two diagonal blocks share identical traces, specifically, $$\mathrm{tr}\left(\textbf{U}\left(f(\mathbf{\Sigma}) + f(-\mathbf{\Sigma})\right)\textbf{U}^T\right) = \mathrm{tr}\left(\textbf{V}\left(f(\mathbf{\Sigma}) + f(-\mathbf{\Sigma})\right)\textbf{V}^T\right).$$
The second equation is valid because of the unbiasedness of the Hutchinson trace estimator, when a Rademacher distributed vector $\textbf{z}_{(n_1)} \in \mathbb{R}^{n_1}$ with mean of $0$ and variance of $1$ \cite{H90} is utilized.

As suggested in Theorem \ref{prop:2}, we consider employing the partial Rademacher vector (denoted by $\tilde{\textbf{z}}$) instead of the complete Rademacher vector to guarantee the symmetry of the quadrature rule. One may set upper partial Rademacher vector $\tilde{\textbf{z}} = \begin{bmatrix}
    \textbf{z}_{(n_1)} ; \textbf{0}_{(n_1)}
\end{bmatrix}$ or lower vector $\tilde{\textbf{z}} = \begin{bmatrix}
    \textbf{0}_{(n_1)}; \textbf{z}_{(n_1)}
\end{bmatrix}$ since the diagonal blocks have the same trace. The expectation of the corresponding quadratic form is
$$
    \mathbb{E}\left[\tilde{\textbf{z}}^T f(\textbf{A}) \tilde{\textbf{z}}\right]  = \frac{1}{2} \mathbb{E}\left[\textbf{z}_{(n_1)}^T \textbf{U} (f(\mathbf{\Sigma}) + f(-\mathbf{\Sigma}))\textbf{U}^T \textbf{z}_{(n_1)} \right] = \frac{1}{2} \mathrm{tr}(f(\textbf{A})).
$$
This indicates that the estimator of $\mathrm{tr}(f(\textbf{A}))$ by the partial Rademacher vector requires doubling but retains unbiasedness, which reads
\begin{equation}
    \label{eq:ste0}
    \mathrm{tr}(f(\textbf{A}))^{\dagger}=\frac{2}{N}\sum_{k=1}^N\tilde{\textbf{z}}_k^Tf(\textbf{A})\tilde{\textbf{z}}_k. 
\end{equation}

\subsection{Case 2: $n_1 \neq n_2$}
Suppose the rank of $\textbf{A}$ is $r$. Based on Theorem \ref{thm:m=n} and under the assumption that $\textbf{B} = \begin{bmatrix}
    \textbf{U}_1 & \textbf{U}_2
\end{bmatrix}\begin{bmatrix}
    \mathbf{\Sigma}_1 & \textbf{O} \\
    \textbf{O} & \textbf{O}
\end{bmatrix}\begin{bmatrix}
    \textbf{V}_1^T \\ \textbf{V}_2^T
\end{bmatrix} \in \mathbb{R}^{n_1 \times n_2}$, where $\textbf{U}_1 \in \mathbb{R}^{n_1 \times r}, \textbf{U}_2\in \mathbb{R}^{n_1 \times (n_1-r)}, \textbf{V}_1 \in \mathbb{R}^{n_2 \times r}, \textbf{V}_2 \in \mathbb{R}^{n_2\times(n_2 - r)}, \mathbf{\Sigma}_1 \in \mathbb{R}^{r\times r}$ and $\textbf{O}$ denote zero matrices of appropriate size, the diagonal blocks of $f(\textbf{A})$ are
$$ \frac{1}{2}\textbf{U}_1(f(\mathbf{\Sigma}_1) + f(-\mathbf{\Sigma}_1))\textbf{U}_1^T + \textbf{U}_2f(\textbf{O}_{(n_1-r)})\textbf{U}_2^T$$
and
$$ \frac{1}{2}\textbf{V}_1(f(\mathbf{\Sigma}_1) + f(-\mathbf{\Sigma}_1))\textbf{V}_1^T + \textbf{V}_2 f(\textbf{O}_{(n_2-r)}) \textbf{V}_2^T.$$

Then, owing to the cyclic property of the trace,
$$ \begin{aligned}
    \mathrm{tr}(f(\textbf{A})) &= 2\cdot \frac{1}{2} \cdot \mathrm{tr}\left( \textbf{V}_1(f(\mathbf{\Sigma}_1) + f(-\mathbf{\Sigma}_1))\textbf{V}_1^T\right) + \mathrm{tr}\left(\textbf{U}_2 f(\textbf{O}_{(n_1-r)}) \textbf{U}_2^T\right) \\
    & \quad+ \mathrm{tr}\left(\textbf{V}_2 f(\textbf{O}_{(n_2-r)}) \textbf{V}_2^T\right) \\
    & = \mathrm{tr}\left(f(\mathbf{\Sigma}_1) + f(-\mathbf{\Sigma}_1)\right) + (n_1+n_2-2r)\cdot f(0).
\end{aligned} $$
Similar to the mathematical derivation in the previous case, if we use a partial Rademacher vector $\tilde{\textbf{z}} = \begin{bmatrix}
    \textbf{0}_{(n_1)}; \textbf{z}_{(n_2)}
\end{bmatrix}$ with $\textbf{z}_{(n_2)} \in \mathbb{R}^{n_2}$, the double expectation of quadratic form is
$$ \begin{aligned}
2\mathbb{E} \left[\tilde{\textbf{z}}^Tf(\textbf{A})\tilde{\textbf{z}}\right]  &= \mathbb{E}\left[\textbf{z}_{(n_2)}^T\textbf{V}_1\left(f(\mathbf{\Sigma}_1) + f(-\mathbf{\Sigma}_1)\right)\textbf{V}_1^T\textbf{z}_{(n_2)} + 2\textbf{z}_{(n_2)}^T \textbf{V}_2 f(\textbf{O}_{(n_2-r)})\textbf{V}_2^T\textbf{z}_{(n_2)}\right] \\
    & = \mathrm{tr}\left(f(\mathbf{\Sigma}_1)+f(-\mathbf{\Sigma}_1)\right) + (2n_2-2r)\cdot f(0).
\end{aligned}$$
Then, 
$$\mathrm{tr}(f(\textbf{A})) = 2\mathbb{E}\left[\tilde{\textbf{z}}^T f(\textbf{A}) \tilde{\textbf{z}}\right] + (n_1 - n_2)\cdot f(0).$$
If $\tilde{\textbf{z}} = \begin{bmatrix}
    \textbf{z}_{(n_1)}; \textbf{0}_{(n_2)}
\end{bmatrix}$, then
$$\mathrm{tr}(f(\textbf{A})) = 2\mathbb{E}\left[\tilde{\textbf{z}}^T f(\textbf{A}) \tilde{\textbf{z}}\right] - (n_1 - n_2)\cdot f(0).$$
This shows that an unbiased estimate of $\mathrm{tr}(f(\textbf{A}))$ can be obtained by first randomly generating a partial Rademacher vector in the Lanczos process with $\textbf{A}$ and then doubling the results of the Lanczos quadrature and adding/subtracting a constant term.

In general, for any $r \le \min\{n_1, n_2\}$, two stochastic trace estimators with $N$ randomly generated upper partial Rademacher vectors $\tilde{\textbf{z}} = \begin{bmatrix}
    \textbf{z}_{(n_1)}; \textbf{0}_{(n_2)}
\end{bmatrix}$ 
\begin{equation}
    \label{eq:ste1}
    \mathrm{tr}(f(\textbf{A}))^{\dagger} = \frac{2}{N}\sum_{i=1}^N \tilde{\textbf{z}}_i^T f(\textbf{A}) \tilde{\textbf{z}}_i + (n_2 - n_1)\cdot f(0)
\end{equation}
or lower partial Rademacher vectors $\tilde{\textbf{z}} = \begin{bmatrix}
    \textbf{0}_{(n_1)}; \textbf{z}_{(n_2)}
\end{bmatrix}$
\begin{equation}
    \label{eq:ste2}
    \mathrm{tr}(f(\textbf{A}))^{\dagger} = \frac{2}{N}\sum_{i=1}^N \tilde{\textbf{z}}_i^T f(\textbf{A}) \tilde{\textbf{z}}_i + (n_1 - n_2)\cdot f(0)
\end{equation}
are unbiased.

%Variance part

\section{Numerical experiments}
\label{sec:experiments}
Numerical experiments are conducted in MATLAB R2025a. Code are available on \url{https://github.com/Justin9872/Tests-SLQ-for-STE}.

\subsection{Test on Theorem \ref{Thm:main}}
\label{sec:TestThmMain}
The first three cases study the reproducible matrix $\textbf{A} = \textbf{H} \mathbf{\Lambda} \textbf{H}^T \in \mathbb{R}^{n\times n}$ with the prescribed diagonal matrix 
$$\mathbf{\Lambda} = \mathrm{diag}(\lambda_1, \lambda_2, \cdots, \lambda_n),$$
and the Householder matrix is constructed according to \cite{zhu09}
$$ \textbf{H} = \textbf{I} - \frac{2}{n}(\mathbf{1} \mathbf{1}^T), $$
where $n = 50$ is chosen for illustration. The fourth case focuses on the \texttt{nd3k} matrix from SuiteSparse Matrix Collection \cite{DH11}. The last two cases study the supra-adjacency matrix based on the \texttt{email} dataset \cite{DH11,PBL17} and the \texttt{web-cs-Stanford} dataset \cite{DH11}. Details of these 6 cases are shown in the following bullet points and Table \ref{Tab:4cases}.
\begin{itemize}
    \item[-] Case 1: $\{\lambda_i\}_{i=1}^{50} = \{i/50\}_{i=1}^{50}$, $\textbf{v} = \mathbf{1}/\sqrt{50}$;
    \item[-] Case 2: $\{\lambda_i\}_{i=1}^{50} = \{1/(51-i)\}_{i=1}^{50}$, $\textbf{v} = \mathbf{1}/\sqrt{50}$;
    \item[-] Case 3: $\{\lambda_i\}_{i=1}^{50} = \{i/50\}_{i=1}^{50}$, $\textbf{v} = (1,2,\cdots,50)^T/\Arrowvert (1,2,\cdots,50)^T \Arrowvert$;
    \item[-] Case 4: \texttt{nd3k} matrix, $\textbf{v} = (1,\cdots,1,-1,\cdots,-1)^T/\sqrt{9000} \in \mathbb{R}^{9000}$;
    \item[-] Case 5: \texttt{email} supra-adjacency matrix, $\textbf{v} = (1,\ldots,1,0,\ldots,0)^T/\sqrt{1005} \in \mathbb{R}^{2010}$;
    \item[-] Case 6: \texttt{web-cs-Stanford} supra-adjacency matrix, $\textbf{v} = (0,\ldots,0,1,\ldots,1)^T/\sqrt{9914} \in \mathbb{R}^{19828}$.
\end{itemize} 
\begin{table}[htbp]
    \caption{Details of the 6 cases with different eigenvalue distributions and starting vectors}
    \label{Tab:4cases}
    \begin{tabular}{|c|c|c|c|c|c|c|}
        \hline & Case 1 & Case 2 & Case 3 & Case 4 & Case 5 & Case 6\\ \hline
        symmetric eigenvalues of $\textbf{A}$? & Yes & No & Yes & No & Yes & Yes   \\ \hline
        absolute palindrome $\boldsymbol{\mu}^{(1)}$?  & Yes & Yes & No & No & Yes & Yes  \\ \hline
        symmetric Ritz values? & Yes & No & No & No & Yes & Yes \\ \hline
    \end{tabular}
\end{table}
Recall that the figure of $\mu(t)$ would be centro symmetric about $(\bar{\lambda}, \mu(\bar{\lambda}))$ when $\textbf{A}$ has a symmetric eigenvalue distribution and $\boldsymbol{\mu}^{(1)}$ is an absolute palindrome. Figure \ref{Fig:Test} shows the shapes of $\mu(t)$ for the 6 cases and the corresponding $10$ Ritz values generated by the Lanczos iteration, offering a visual representation that substantiates the validity of Theorem \ref{Thm:main} to a certain degree.

\begin{figure}[htbp]
    \centering
    \subfloat[Case 1]{\label{Fig:case1}\includegraphics[width = 0.5\textwidth]{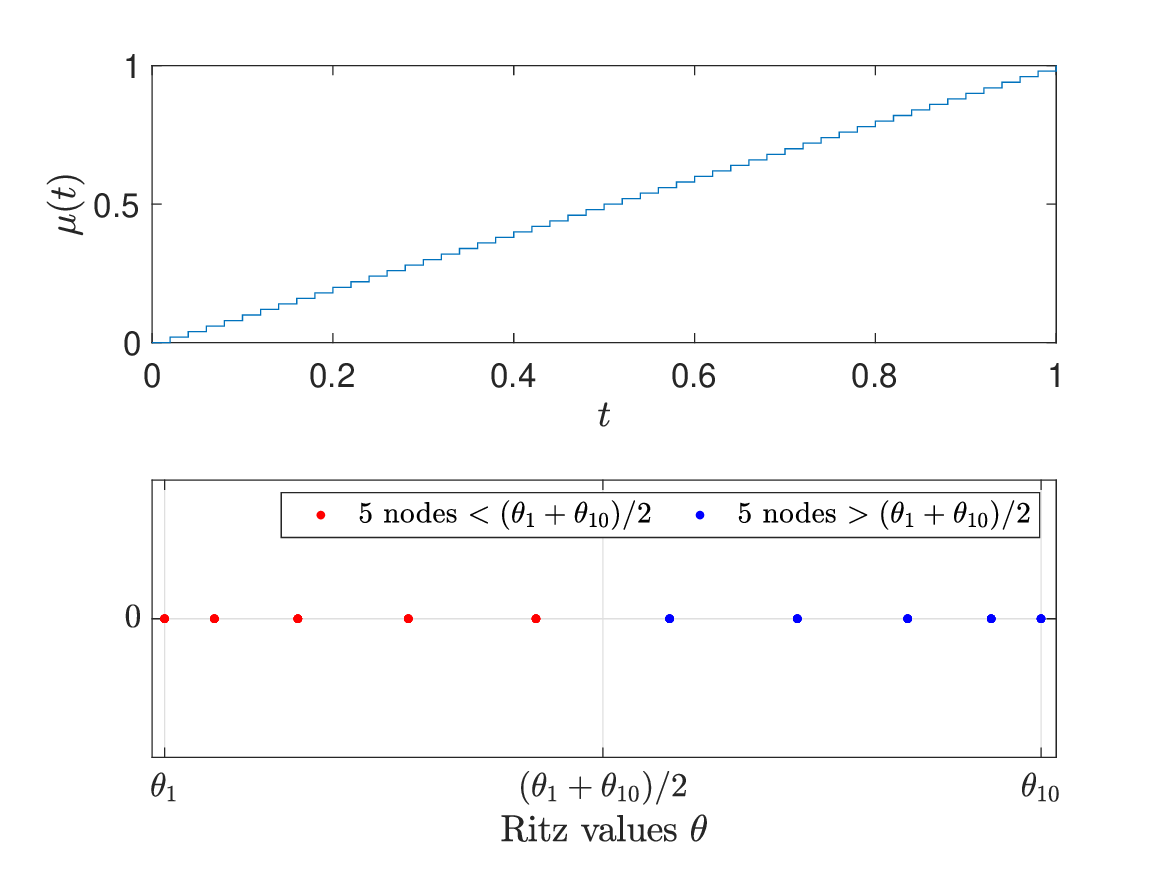}}
    \subfloat[Case 2]{\label{Fig:case2}\includegraphics[width = 0.5\textwidth]{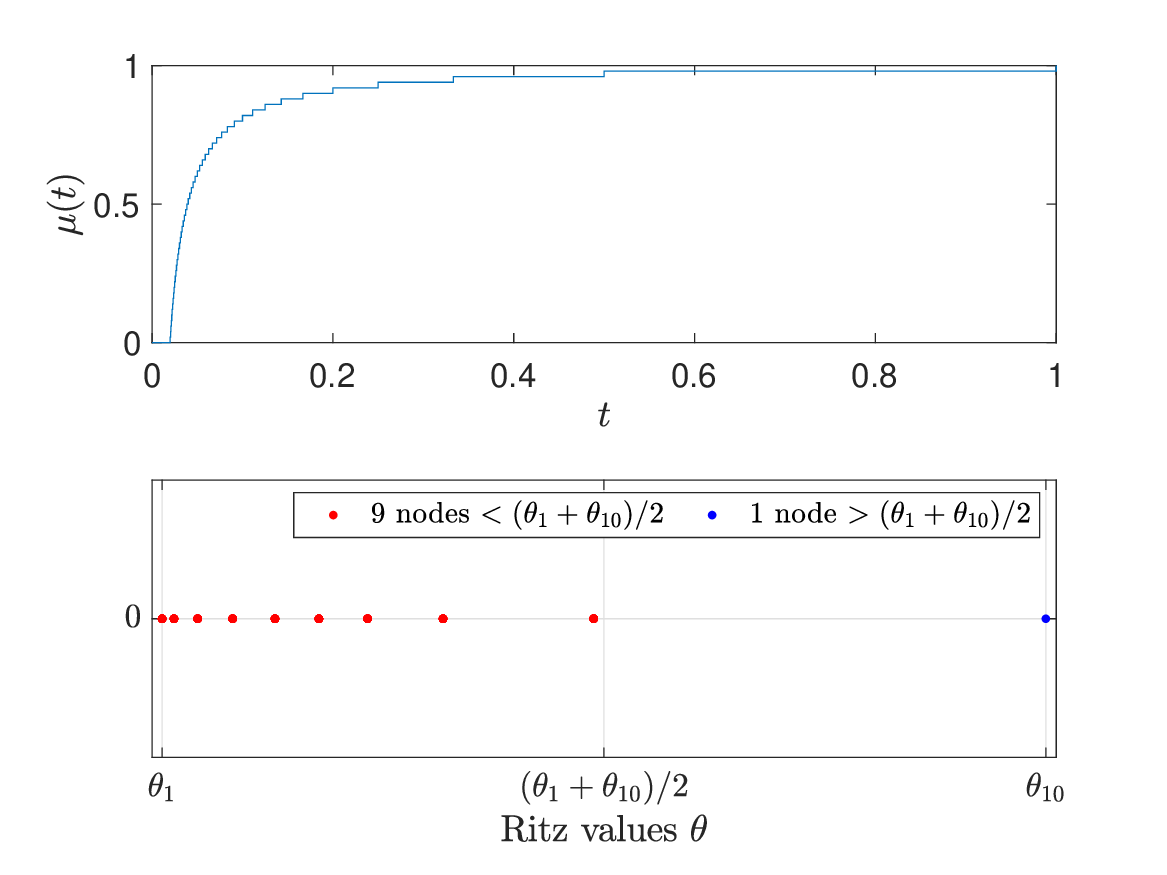}}\\
    \subfloat[Case 3]{\label{Fig:case3}\includegraphics[width = 0.5\textwidth]{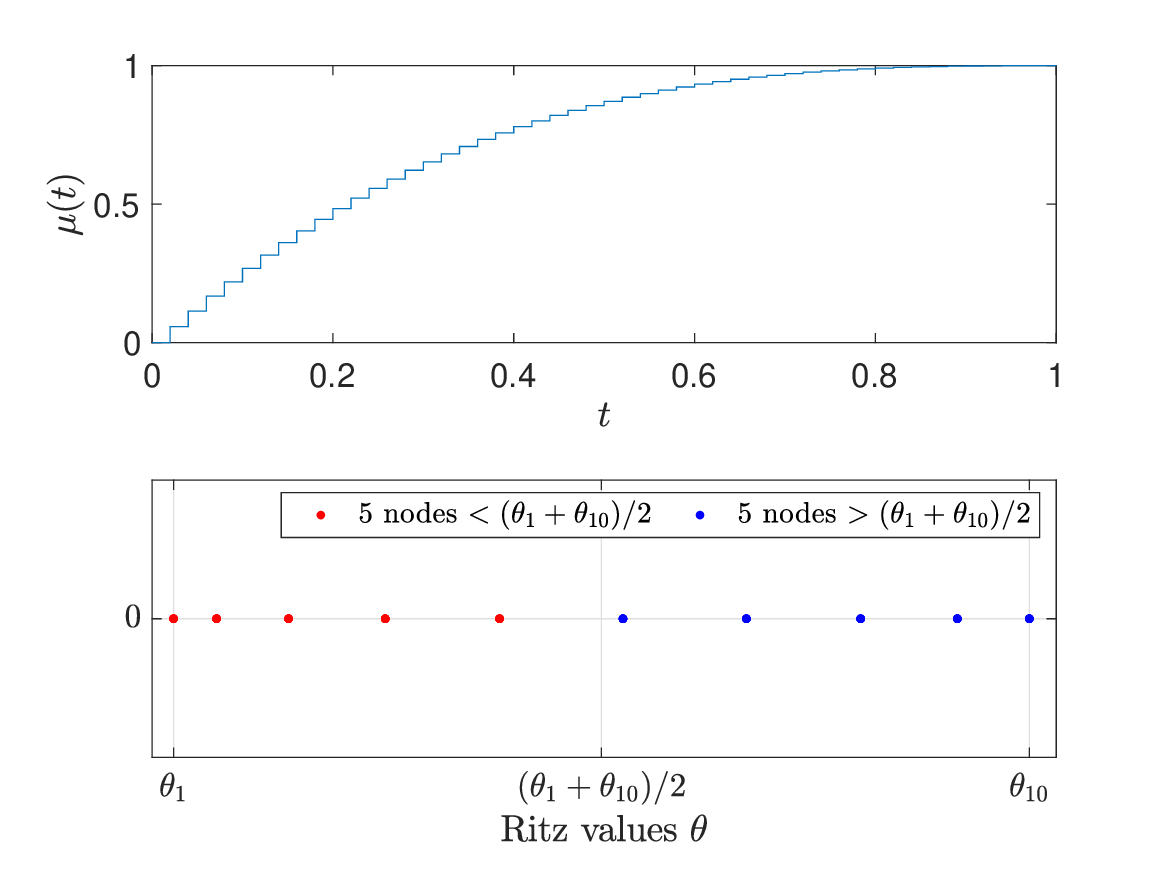}}
    \subfloat[Case 4]{\label{Fig:case4}\includegraphics[width = 0.5\textwidth]{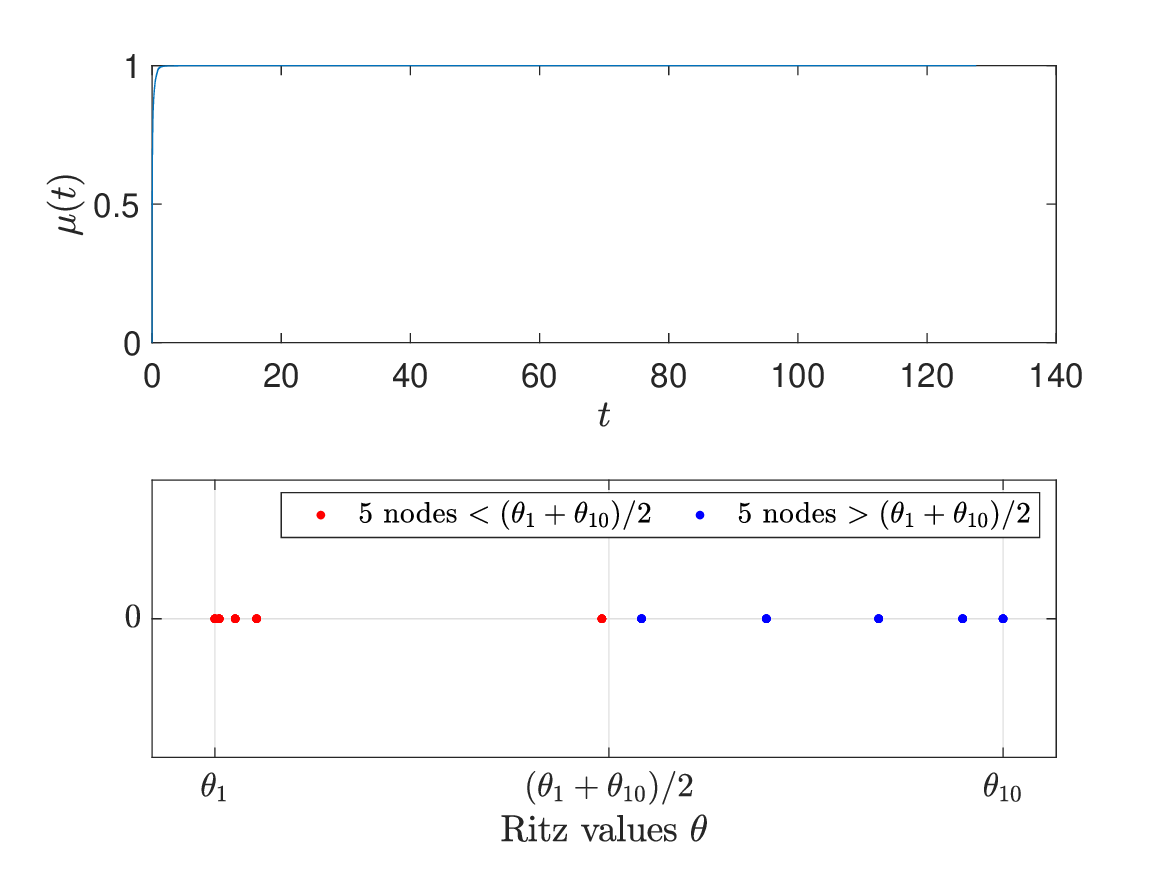}}\\
    \subfloat[Case 5]{\label{Fig:case5}\includegraphics[width = 0.5\textwidth]{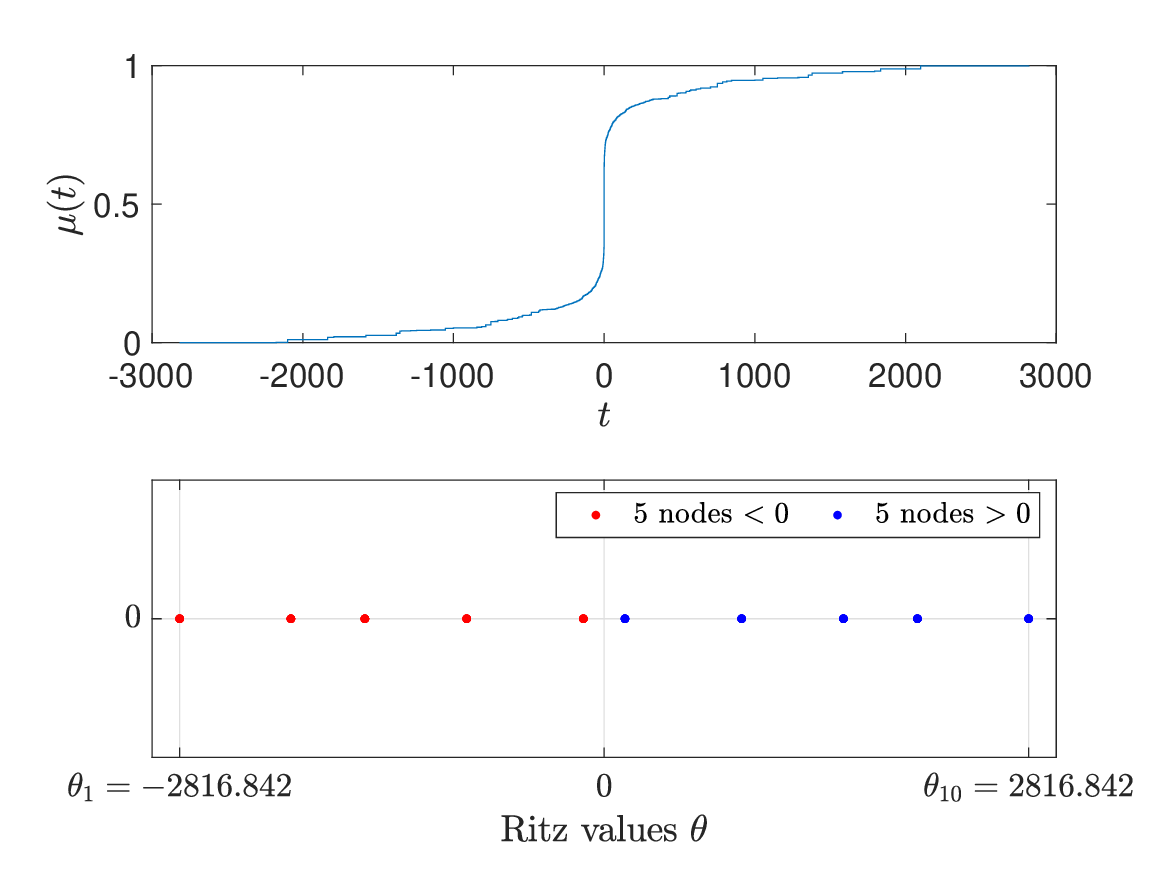}}
    \subfloat[Case 6]{\label{Fig:case6}\includegraphics[width = 0.5\textwidth]{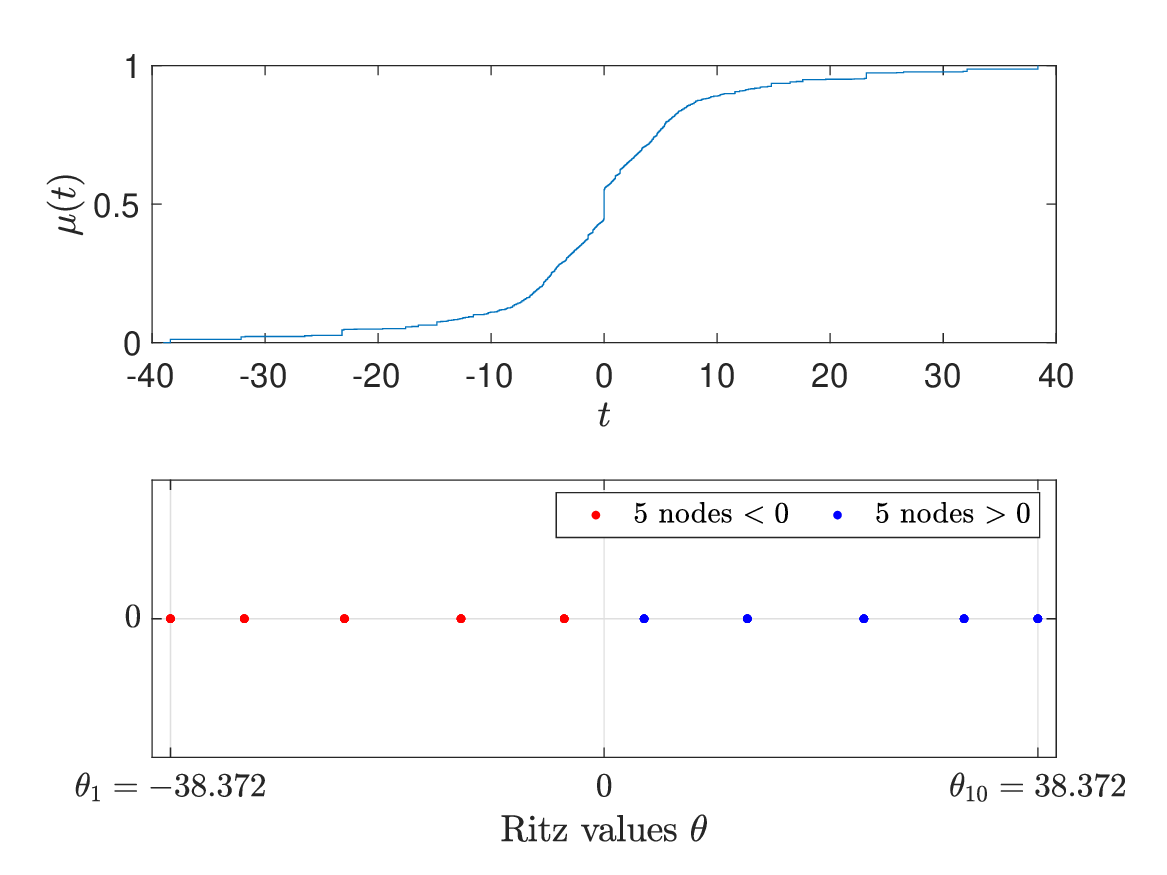}}\\
    \caption{Plots of discrete measure functions $\mu(t)$ \eqref{eq:measure_f} and the locations of Ritz values in six cases}
    \label{Fig:Test}
\end{figure}

\subsection{Test on Estrada index estimation}
\label{sec:estrada}
Numerical experiments are conducted to show that for estimating the Estrada index of directed or bipartite graphs, approximating quadratic forms in the stochastic trace estimator $\mathrm{tr}(f(\textbf{A}))^{\dagger}$ with upper or lower partial Rademacher distributed initial vectors (see \eqref{eq:ste1} and \eqref{eq:ste2}) results in lower variance compared to full Rademacher initial vectors. Furthermore, under identical computational budgets, the proposed estimator achieves a lower relative error than the Hutch++ benchmark \cite{M21}.

In the following tests, we first set $N = 1, m = 100$ for each run of the stochastic Lanczos quadrature method to control the variance-reduction effect of the Monte Carlo simulation, and observe the variance of quadratic form estimations. Then we test the three estimators as the number of queries increases. We conduct 100 trials for each choice of $N\in \{30,40,\ldots,180\}$ and fixed $m = 100$, and plot the median, the $25^{th}$ and $75^{th}$ percentile relative errors.

\subsubsection{Synthetic matrix}
We first consider a synthetic Jordan-Wielandt matrix with $\textbf{B} = \textbf{U}\mathbf{\Sigma} \textbf{V}^T \in \mathbb{R}^{1000\times 1000}$. $\textbf{U}$ and $\textbf{V}$ are obtained by generating Gaussian distributed matrices and then orthogonalizing them, i.e., $\texttt{U = orth(randn(1000,1000));} \texttt{V = orth(randn(1000,1000))}$. The diagonal matrix $\mathbf{\Sigma}$ is also randomly generated by $\texttt{Sigma = diag(randn(1000,1))}$. Then, the tested synthetic matrix is created by 
$\texttt{A = [zeros(1000,1000) U*Sigma*V'; V*Sigma*U' zeros(1000,1000)]}$. 

We randomly generate one normalized Rademacher vector and two normalized partial Rademacher vectors (with a reproducible seed generator) as the initial vectors for 100 times. Then we compare the variances of quadratic form estimations in these estimators for $\mathrm{tr}(e^{\gamma \textbf{A}})$ with $\gamma = 1$, $N = 1$ and $m = 100$ Lanczos iterations. Table \ref{table:syn} and Figure \ref{fig:syn} illustrate that partial Rademacher initial vectors result in lower variances than the fully Rademacher distributed vectors within the same computational budget. 

\begin{figure}[htbp]
    \centering
    \includegraphics[width = 90mm]{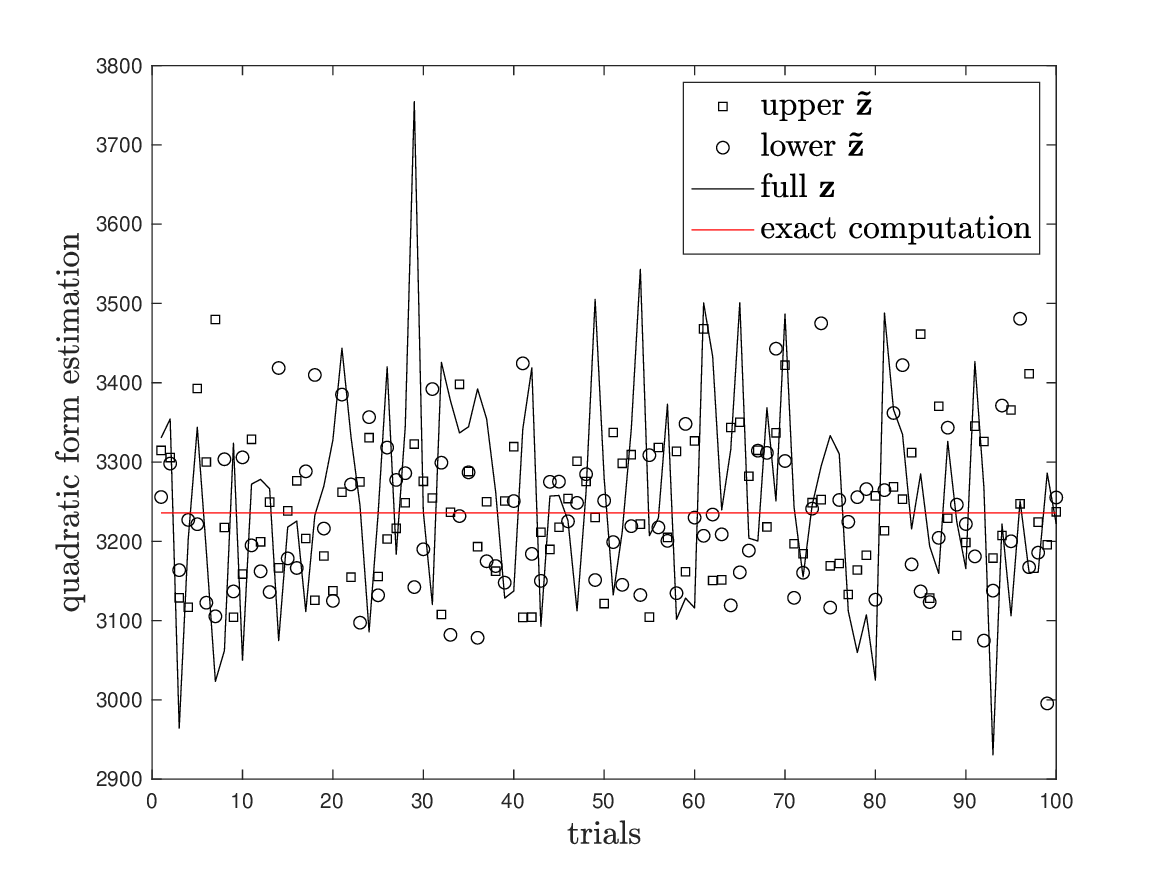}
    \caption{One hundred trials of quadratic form estimation in approximating $\mathrm{tr}(e^{\gamma \textbf{A}})$ with $\gamma = 1$, $N = 1$, $m=100$, synthetic Jordan-Wielandt matrix $\textbf{A}$ and different initial vectors. $\tilde{\textbf{z}}$ denote upper and lower partial Rademacher vectors, whereas all the elements of $\textbf{z}$ are Rademacher distributed}
    \label{fig:syn}
\end{figure}

\begin{table}[htbp]
    \caption{Variances of the quadratic form estimations in Figure \ref{fig:syn}}
    \label{table:syn}
    \begin{tabular}{@{}llll@{}}
        \toprule
            & lower partial Rademacher $\tilde{\textbf{z}}$ & upper partial Rademacher $\tilde{\textbf{z}}$ & fully Rademacher $\textbf{z}$ \\ \midrule
        Variance & 88.04 & 95.98 & 134.96 \\ \bottomrule
    \end{tabular}
\end{table}

Finally, we compare the accuracy of three different estimators: the Hutchinson trace estimator, the Hutchinson trace estimator with upper partial Rademacher vector, and the Hutch++ estimator. For each choice of $N\in\{30,40,\ldots,180\}$, the three estimators are run for 100 trials, with fixed $m=100$. Figure \ref{fig:syn_re} shows that the application of partial Rademacher starting vectors in the Hutchinson estimator results in higher accuracy than other estimators within the same number of matrix-vector multiplications.

\begin{figure}[htbp]
    \centering
    \includegraphics[width=120mm]{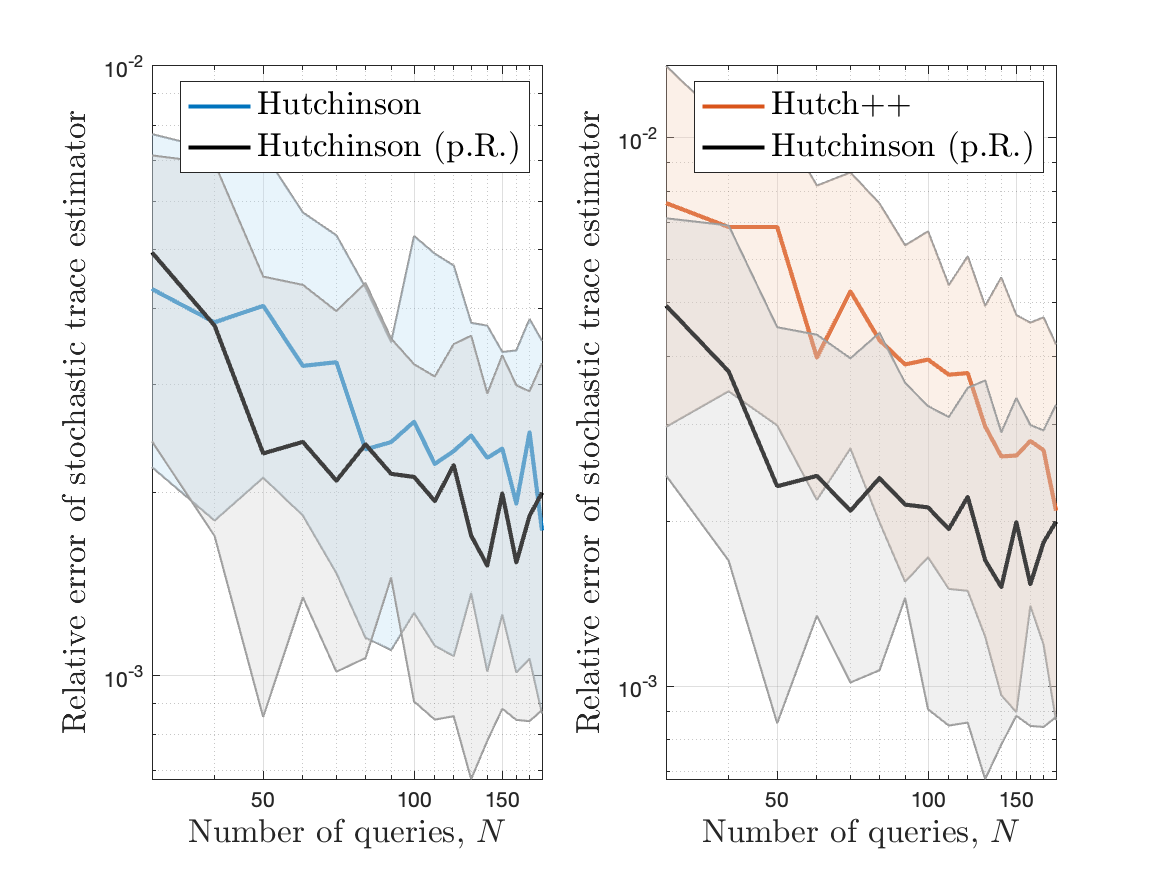}
    \caption{Relative error of stochastic Lanczos quadrature estimators against different numbers of queries $N\in\{30,40,\ldots,180\}$. In order to compare the accuracy of three estimators, i.e., the Hutchinson trace estimator, the Hutchinson trace estimator with upper partial Rademacher vector (denoted by p.R.), and the Hutch++ trace estimator, we record the median, the $25^{th}$ and $75^{th}$ percentile relative errors of approximating $\mathrm{tr}(e^{\gamma \textbf{A}})$ after 100 trials per parameter set. The shaded area around each curve is bounded by the $25^{th}$ and $75^{th}$ percentile curves of the corresponding estimator. Note that $\textbf{A}$ is a synthetic Jordan-Wielandt matrix, $\gamma = 1$ and $m=100$ are fixed during the test.}
    \label{fig:syn_re}
\end{figure}

\subsubsection{\texttt{email} dataset}
A further example is based on a directed network example of 1005 nodes and 24929 edges from the \texttt{email} dataset \cite{DH11,PBL17}. The adjacency matrix $\textbf{B}$ of this directed network is nonsymmetric, so we build a supra-adjacency matrix in the form of \eqref{Eq:A}. The estimation of the Estrada index $\mathrm{tr}(e^{\gamma \textbf{A}})$ with parameter $\gamma = 0.5/\lambda_{\max}$ is of interest. We again set $N = 1, m = 100$ for the stochastic Lanczos quadrature method and compare the estimations with two types of partial Rademacher vectors and fully distributed Rademacher vectors after 100 trials. Table \ref{table:email} and Figure \ref{fig:email} reflect the effect of variance reduction by using partial Rademacher vectors.

\begin{figure}[htbp]
    \centering
    \includegraphics[width = 90mm]{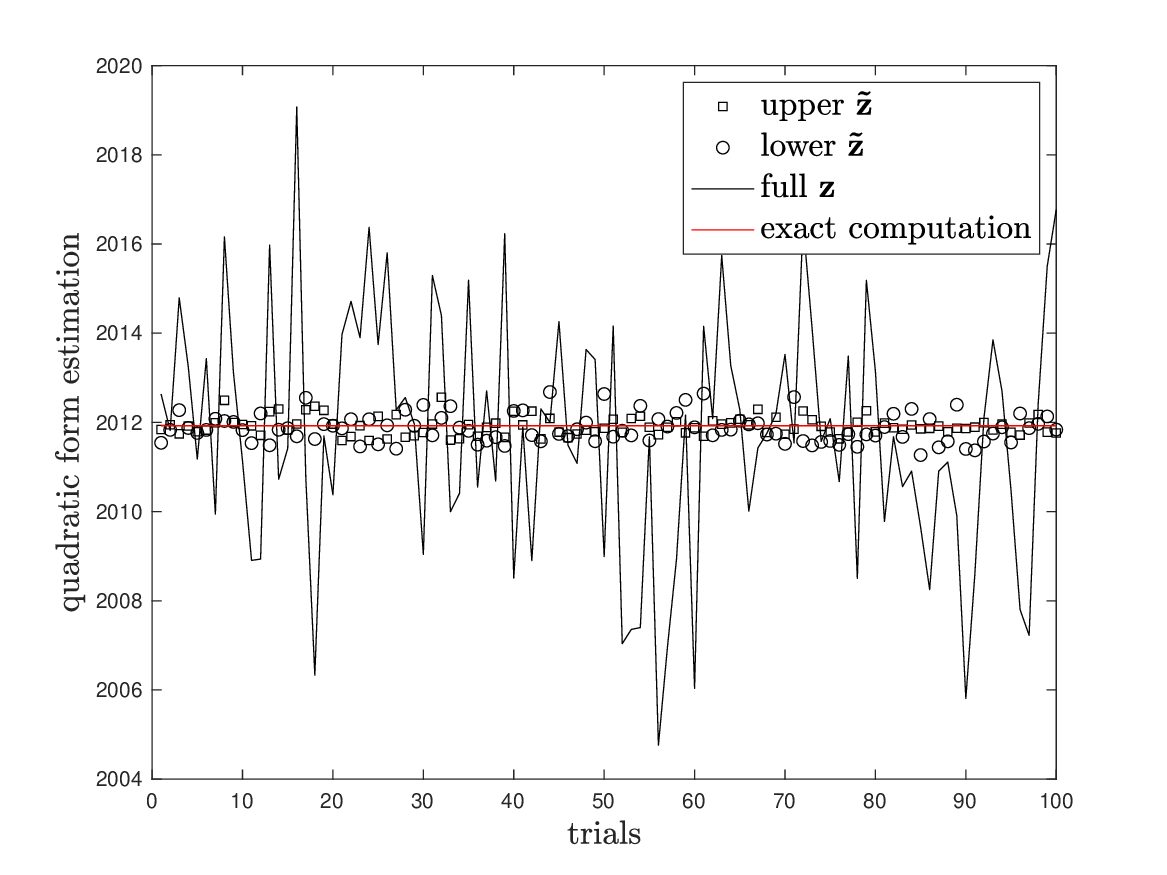}
    \caption{One hundred trials of quadratic form estimation in approximating $\mathrm{tr}(e^{\gamma \textbf{A}})$ with $\gamma = 0.5/\lambda_{\max}$, $N=1$, $m=100$, supra-adjacency matrix based on the \texttt{email} dataset \cite{DH11,PBL17} and different initial vectors. $\tilde{\textbf{z}}$ denote upper and lower partial Rademacher vectors, whereas all the elements of $\textbf{z}$ are Rademacher distributed}
    \label{fig:email}
\end{figure}

\begin{table}[htbp]
    \caption{Variances of the quadratic form estimations in Figure \ref{fig:email}}
    \label{table:email}

    \begin{tabular}{@{}llll@{}}
        \toprule
           & lower partial Rademacher $\tilde{\textbf{z}}$ & upper partial Rademacher $\tilde{\textbf{z}}$ & fully Rademacher $\textbf{z}$  \\ \midrule Variance & 0.20                             & 0.32  & 2.75 \\ \bottomrule
    \end{tabular}
\end{table}

Similar to the previous test on the synthetic matrix, we compare the performances of the Hutchinson estimator and the Hutch++ on approximating $\mathrm{tr}(e^{0.5\textbf{A}}/\lambda_{\max})$ of the supra-adjacency matrix $\textbf{A}$, with fixed $m=100$ and $N = 30,40,\ldots,180$. Figure \ref{fig:email_re} suggests the Hutchinson trace estimator with partial Rademacher vectors for this standardized matrix. Note that the Hutch++ estimator does not perform well in this case, since $e^{0.5\textbf{A}/\lambda_{\max}}$ has a relatively flat spectrum in $\left[e^{-0.5},e^{0.5}\right]$ after scaling.

\begin{figure}[htbp]
    \centering
    \includegraphics[width = 90mm]{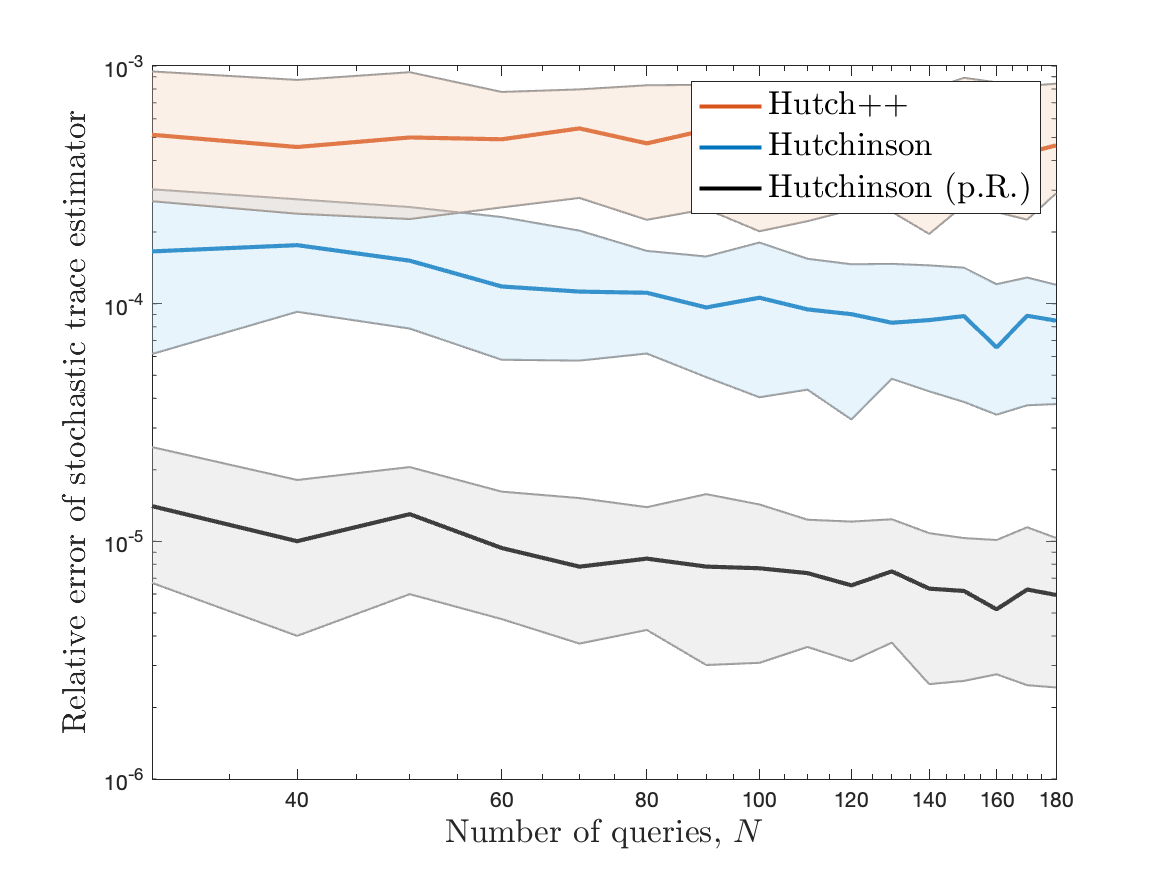}
    \caption{Relative error of stochastic Lanczos quadrature estimators against different numbers of queries $N\in\{30,40,\ldots,180\}$. In order to compare the accuracy of three estimators, i.e., the Hutchinson trace estimator, the Hutchinson trace estimator with upper partial Rademacher vector (denoted by p.R.), and the Hutch++ trace estimator, we record the median, the $25^{th}$ and $75^{th}$ percentile relative errors of approximating $\mathrm{tr}(e^{\gamma \textbf{A}})$ after 100 trials per parameter set. The shaded area around each curve is bounded by the $25^{th}$ and $75^{th}$ percentile curves of the corresponding estimator. Note that $\textbf{A}$ is the supra-adjacency matrix based on the \texttt{email} dataset \cite{DH11,PBL17}, $\gamma = 0.5/\lambda_{\max}$ and $m=100$ are fixed during the test.}
    \label{fig:email_re}
\end{figure}

\subsubsection{\texttt{Notre Dame networks} dataset}
We also study a bipartite graph example with 392400 players and 127823 movies from the Group Barabasi, \texttt{Notre Dame networks} dataset \cite{AJB99,DH11}. The full adjacency matrix of this network is also of the Jordan-Wielandt type, with the biadjacency matrix $\textbf{B} \in \mathbb{R}^{392400\times 127823}$. Similar to the previous two examples, with 100 iterations of the Lanczos method, tests are carried out on the two proposed trace estimators \eqref{eq:ste1}, \eqref{eq:ste2} and the Hutchinson trace estimator \eqref{eq:Hutchinson} for $\mathrm{tr}(e^{\gamma \textbf{A}})$, $\gamma = 1/\lambda_{\max}$. 

Owing to the limitations of memory and storage, the exact value of the Estrada index for the adjacency matrix of \texttt{Notre Dame networks} dataset is difficult to compute. Due to the unknown exact value, it is unavailable to record the relative error of approximations by different trace estimators. The results in Table \ref{table:movie} and Figure \ref{fig:movie} demonstrate that the use of a lower/upper partial Rademacher vector $\tilde{\textbf{z}}$ helps estimate the Estrada index with higher stability.

\begin{figure}[htbp]
    \centering
    \includegraphics[width = 90mm]{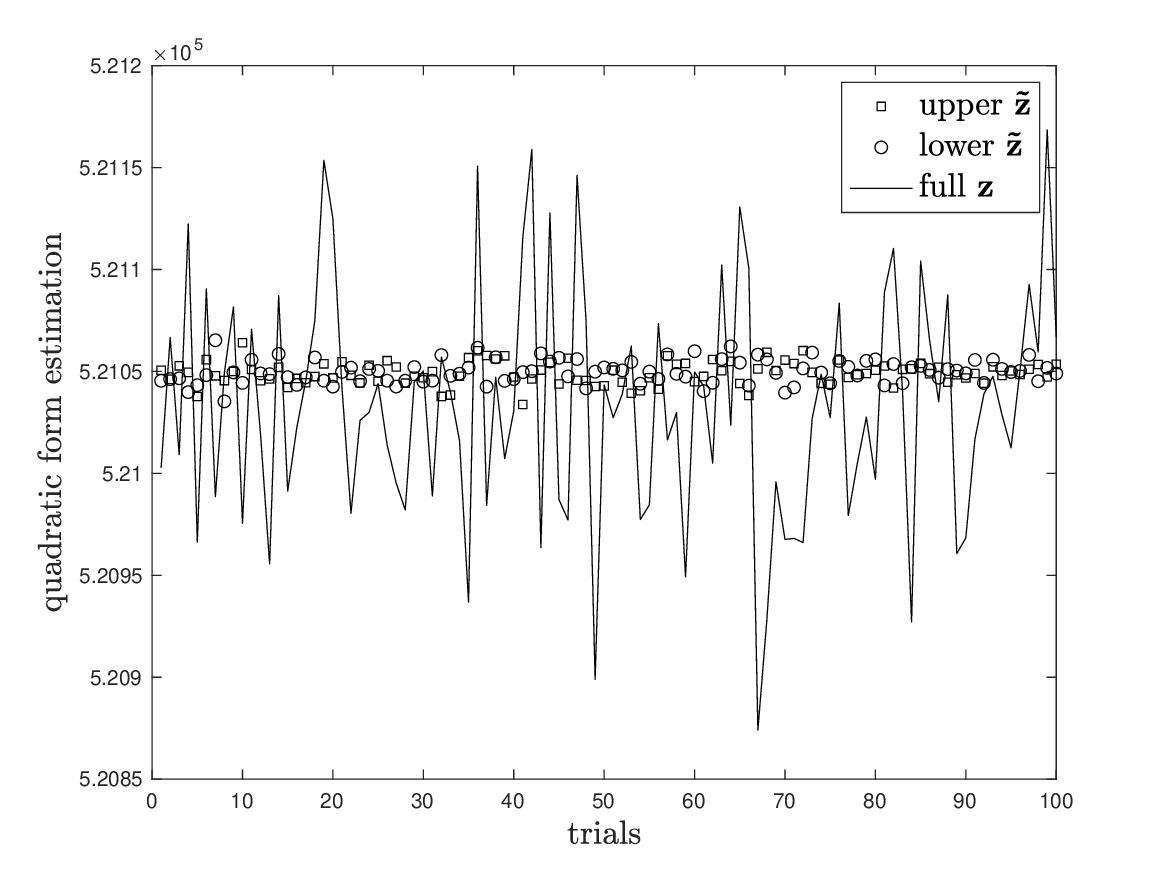}
    \caption{One hundred trials of quadratic form estimation in approximating $\mathrm{tr}(e^{\gamma \textbf{A}})$ with $\gamma = 1/\lambda_{\max}$, $N = 1$, $m=100$, full adjacency matrix of the \texttt{Notre Dame networks} dataset \cite{AJB99,DH11} and different initial vectors. $\tilde{\textbf{z}}$ denote upper and lower partial Rademacher vectors, whereas all the elements of $\textbf{z}$ are Rademacher distributed}
    \label{fig:movie}
\end{figure}

\begin{table}[htbp]
    \caption{Variances of the quadratic form estimations in Figure \ref{fig:movie}}
    \label{table:movie}

    \begin{tabular}{@{}llll@{}}
        \toprule
           & lower partial Rademacher $\tilde{\textbf{z}}$ & upper partial Rademacher $\tilde{\textbf{z}}$ & fully Rademacher $\textbf{z}$  \\ \midrule Variance & 5.56                            & 5.67 & 57.96 \\ \bottomrule
    \end{tabular}

\end{table}

\section{Concluding discussion}

\label{sec:cr}

Symmetric Lanczos quadrature rules enjoy a higher convergence rate than asymmetric rules in estimating Riemann-Stieltjes integrals. There was no sufficient and necessary condition to guarantee a symmetric Lanczos quadrature rule until this paper. This paper proves that for the class of Jordan-Wielandt matrices, with a careful choice of the starting vector for the Lanczos process, one can realize a symmetric quadrature rule. Such a favorable Lanczos quadrature can be used in estimating the Estrada index of bipartite or directed graphs. We only move forward a bit on finding a sufficient and necessary condition of a symmetric Lanczos quadrature for a special class of matrices. For general matrices, this problem remains open and more challenging.
\vspace{0.5cm}

\noindent {\bf Acknowledgments.} The research is supported by Natural Science Foundation of China (12271047); National Key Technologies Research and Development Program(2025YFG0202100); Guangdong Provincial Key Laboratory of Interdisciplinary Research and Application for Data Science, Beijing Normal-Hong Kong Baptist University (2022B1212010006); BNBU research grant (UICR0400036-21C, UICR0400008-21; R04202405-21); Guangdong College Enhancement and Innovation Program (2021ZDZX1046). We are grateful to Professor Zhongxiao Jia of Tsinghua University and Professor James Lambers of University of Southern Mississippi for critical reading of the manuscript and providing useful feedback. We also appreciate the feedback provided by the reviewers.

\section*{Appendix}
\begin{center}
		\begin{longtable}{lll}
			\caption{Table of Notation} \label{tab:notation} \\
			
			\hline \multicolumn{1}{l}{\textbf{Symbols}} & \multicolumn{1}{l}{\textbf{Descriptions}} \\ \hline 
			\endfirsthead
			
			\multicolumn{2}{l}%
			{{\bfseries \tablename\ \thetable{} -- continued from previous page}} \\
			\hline \multicolumn{1}{l}{\textbf{Symbols}} & \multicolumn{1}{l}{\textbf{Descriptions}} \\
            \hline 
			\endhead
			
			\hline \multicolumn{2}{l}{{Continued on next page}} \\ 
			\endfoot
			
			\hline
			\endlastfoot
			\multicolumn{2}{l}{{Functionals}}   \\
          \hline
          $f(\cdot), g(\cdot)$ & function or matrix function \\
          $\omega(\cdot)$ & weight function \\
          $\mu(\cdot)$ & measure function \\
          $l_k(\cdot)$ & orthogonal polynomial of degree $k$ \\
          \hline
          \multicolumn{2}{l}{{Matrices}}   \\
          \hline
          $\textbf{A}$ & symmetric matrix \\
          $\textbf{Q}$ & orthogonal matrix with eigenvectors of $\textbf{A}$ \\
          $\mathbf{\Lambda}$ & diagonal matrix with eigenvalues of $\textbf{A}$ \\
          $\textbf{B}$ & matrix \\
          $\textbf{U}$ & orthogonal matrices with left singular vectors of $\textbf{B}$ \\
          $\textbf{V}$ & orthogonal matrices of right singular vectors of $\textbf{B}$ \\
          $\mathbf{\Sigma}$ & diagonal matrix with singular values of $\textbf{B}$\\ 
          $\textbf{T}$ & Jacobi matrix (symmetric tridiagonal matrix) \\
          $\mathbf{\Gamma}$ & orthogonal matrix with eigenvectors of $\textbf{T}$ \\
          $\mathbf{\Theta}$ & diagonal matrix with eigenvalues of $\textbf{T}$ \\
          $\textbf{J}$ & bidiagonal matrix \\
          $\textbf{X}$ & orthogonal matrix with left singular vectors of $\textbf{J}$ \\
          $\textbf{Y}$ & orthogonal matrix with right singular vectors of $\textbf{J}$ \\
          $\textbf{I}$ & identity matrix \\
          $\textbf{O}$ & zero matrix \\
          $\textbf{P}$ & permutation matrix (anti-diagonal identity matrix in some cases) \\
          $\textbf{S}$ & signature matrix with diagonal entries $\pm 1$ \\
          $\textbf{H}$ & Householder matrix \\
          $\textbf{[\quad]}_i, \textbf{[\quad]}_{ij}$ & $(i)$-block or $(i,j)$-block of the corresponding matrix \\
          $\textbf{[\quad]}_{(n)}$ & matrix of size $n$ \\
          \hline
          \multicolumn{2}{l}{{Vectors}}   \\
          \hline
          $\textbf{u}, \textbf{v}$ & vectors in the Lanczos process \\
          $\textbf{w}$ & (partial) absolute palindrome \\
          $\boldsymbol{\xi}$ & coordinate vector \\
          $\textbf{e}_i$ & $i^{th}$ column of the identity matrix \\ 
          $\textbf{0}, \textbf{1}$ & vector of all zeros or ones \\
          $\boldsymbol{\mu}$ & measure vector \\
          $\boldsymbol{\psi}$ & eigenvectors of $\textbf{T}$ \\
          $\textbf{x}, \textbf{y}$ & left and right singular vectors of $\textbf{J}$ \\
          $\textbf{z}, \tilde{\textbf{z}}$ & Rademacher vector, partial Rademacher vector \\
          $\tilde{\textbf{z}}$ & partial Rademacher vector \\
          $|\cdot|$ & vector with absolute elements \\
          $\textbf{[\;]}_{i}$ & $i^{th}$ vector in Monte Carlo simulation \\
          $\textbf{[\;]}^{(k)}$ & $k^{th}$ vector in the Lanczos iteration \\
          $\textbf{[\;]}_{(n)}$ & vector of size $n$ \\
          \hline
          \multicolumn{2}{l}{{Scalars}}   \\
          \hline
          $i, j, k, p, q$ & indices or subscripts \\
          $m$ & step of the Lanczos algorithm \\
          $m^*$ & maximum iteration \\
          $n$, $n_1$, $n_2$ & dimension of matrix \\
          $N$ & number of Monte Carlo simulations \\
          $a, b$ & real numbers, left and right endpoints of interval $[a,b]$ \\
          $r, r_{\textbf{B}}$ & rank of matrix\\
          $s$ & shift \\
          $t$ & independent variable with respect to functionals \\
          $\alpha, \beta$ & entries in the Jacobi/bidiagonal matrix \\
          $\gamma$ & constant \\
          $\kappa$ & condition number of matrix \\
          $\theta_k$ & quadrature nodes/eigenvalues of $\textbf{T}$ \\
          $\tau_k$ & quadrature weights \\
          $\lambda_j$ & eigenvalues of $\textbf{A}$ \\
          $\delta$ & singular values of $\textbf{J}$ \\
          $[\;]_j$ & $j^{th}$ element of certain vector \\
          $\epsilon$ & error bound \\
          $\rho$ & elliptical radius \\
          $E_{\rho}$ & open Bernstein ellipse \\
          $M_{\rho}$ & maximum of $|g(\cdot)|$ on $E_\rho$ \\
          $\mathcal{I}$ & Riemann-Stieltjes integral \\
          $\mathcal{I}_m$ & $m$-node Lanczos quadrature \\
          $Q(\textbf{u},f,\textbf{A})$ & quadratic form  \\
          $Q_m(\textbf{u},f,\textbf{A})$ & approximation of quadratic form by $m$-node Lanczos quadrature \\
          $\mathrm{tr}(f(\textbf{A}))$ & trace of matrix function \\
          $\mathrm{tr}(f(\textbf{A}))^{\dagger}$ & stochastic trace estimators with partial Rademacher vectors \\
        			
		\end{longtable}
	\end{center}

%This paper provides sufficient conditions when and on what kinds of matrices the Golub-Welsh algorithm can produce symmetric quadrature rules.  The problem seems trivial initially and turns out to be not so straightforward. 

%This paper stems from discrepancies in the analyses of quadrature error observed between \cite{UCS17} and \cite{CK21}. The difference results from the symmetry of quadrature nodes (Ritz values) in the Lanczos quadrature method. Consequently, we introduce and establish a sufficient condition ensuring this symmetry. Furthermore, for applications involving the estimation of quadratic forms, we investigate the disparity between the minimum Lanczos iterations required to guarantee equivalent approximation accuracy in cases of asymmetric and symmetric quadrature nodes. We also show that in the applications that require the Estrada index of bipartite matrices, initial vectors with half zeros and half Rademacher entries ($\pm 1$) ensure symmetric quadrature nodes and unbiased trace estimators. 

%We expect the emergence of other theories on the symmetry of Ritz values in the Lanczos method, e.g., sufficient and necessary conditions of such symmetry. We believe that based on such elucidations one may correctly develop further analyses on the Lanczos quadrature method \cite{HZY1}, the properties of Ritz values and facilitate the study of other Gauss quadrature rules \cite{WX23}. 

\end{document}